\documentclass[12pt]{article}
\usepackage{amsmath}
\usepackage{amsfonts}
\usepackage{mitpress}
\usepackage{comment}
\usepackage[toc,page, title]{appendix}
\usepackage{graphicx}
\usepackage{color}

\newcommand{\R}{{\mathbf R}}
\newcommand{\Dom}{\mathop{\rm Dom}}
\newcommand{\Domo}{{\Dom_0}}
\newcommand{\spt}{\mathop{\rm spt}}
\newcommand{\f}{{F}}
\newcommand{\p}{{\partial}}
\newcommand{\X}{{X}}
\newcommand{\Y}{{Y}}
\newcommand{\oX}{{\overline X}}
\newcommand{\oY}{{\overline Y}}
\newtheorem{theorem}{Theorem}

\newtheorem{corollary}[theorem]{Corollary}

\newtheorem{definition}[theorem]{Definition}
\newtheorem{example}[theorem]{Example}

\newtheorem{lemma}[theorem]{Lemma}

\newtheorem{proposition}[theorem]{Proposition}
\newtheorem{remark}[theorem]{Remark}

\newenvironment{proof}[1][Proof]{\noindent\textbf{#1.} }{\ \rule{0.5em}{0.5em}}
\newdimen\dummy
\dummy=\oddsidemargin
\newcommand{\citeasnoun}{\cite}
\begin{document}

\title{Multi- to one-dimensional transportation\thanks{%
The authors are grateful to Toronto's Fields' Institute for the Mathematical
Sciences for its kind hospitality during part of this work. RJM acknowledges
partial support of this research by Natural Sciences and Engineering
Research Council of Canada Grant 217006-08 and -15, and by the US
National Science Foundation under Grant No.\ DMS-14401140 while in
residence at the Mathematical Sciences Research Institute in Berkeley CA
during the spring 2016 semester. Chiappori gratefully
acknowledges financial support from the NSF (Award 1124277). Pass is pleased
to acknowledge support from Natural Sciences and Engineering Research
Council of Canada Grant 412779-2012 and a University of Alberta start-up
grant. \hskip0.75in \copyright \today}}
\author{Pierre-Andr\'{e} Chiappori\thanks{%
Department of Economics, Columbia University, New York, USA
pc2167@columbia.edu}, Robert J McCann\thanks{%
Department of Mathematics, University of Toronto, Toronto, Ontario, Canada
mccann@math.toronto.edu} \ and Brendan Pass\thanks{%
Department of Mathematical and Statistical Sciences, University of Alberta,
Edmonton, Alberta, Canada pass@ualberta.ca.}}

\maketitle

\begin{abstract}
We consider the Monge-Kantorovich problem of transporting
a probability density on $\R^m$ to another on the line,  so as to optimize a
given cost function.  We introduce a nestedness criterion relating the cost to
the densities,  under which it becomes possible to solve this problem uniquely,  by constructing an optimal map one level set at a time.  This map is continuous if the
target density has connected support.  We use level-set dynamics to develop
and quantify a local regularity theory for this map and the Kantorovich potentials
solving the dual linear program. We identify obstructions to global regularity
through examples.

More specifically, fix probability densities $f$ and $g$ on open sets $X \subset \R^m$ and $Y \subset \R^n$ with $m\ge n\ge1$.
Consider transporting $f$ onto $g$ so as to minimize the cost $-s(x,y)$.  We give a non-degeneracy condition (a) on $s \in C^{1,1}$ which ensures the set of $x$ paired with [$g$-a.e.] $y\in Y$ lie in a codimension $n$ submanifold of $X$.
Specializing to the case $m>n=1$,  we discover a nestedness criteria relating $s$ to $(f,g)$ which 
allows us to construct a unique optimal solution in the form of a map $F:X \longrightarrow \overline Y$.
When $s\in C^2 \cap W^{3,1}$ and $\log f$ and $\log g$ are bounded,  
the Kantorovich dual potentials $(u,v)$ satisfy $v \in C^{1,1}_{loc}(Y)$,  and the normal velocity $V$ of $F^{-1}(y)$ with respect to changes in $y$ is given by $V(x) = v''(\f(x))-s_{yy}(x,\f(x))$. Positivity (b) of $V$ locally implies a Lipschitz bound on $\f$; 
moreover, $v \in C^2$ if ${F^{-1}(y)}$ intersects $\partial X \in C^1$ transversally (c).
On subsets where (a)-(c) can be be quantified,  for each integer $r \ge1$ the norms of $u,v \in C^{r+1,1}$ and  $F \in C^{r,1}$
are controlled by these bounds, $\|\log f,\log g, \partial X \|_{C^{r-1,1}}, \|\p X\|_{C^{1,1}}$, $\|s\|_{C^{r+1,1}}$, 
and the smallness of $F^{-1}(y)$.  We give examples showing regularity extends from $X$ to part of $\oX$,
but not from $Y$ to $\oY$.  We also show that when $s$ remains nested for all $(f,g)$,  
the problem in $\R^m \times \R$ reduces to a supermodular 
problem in $\R \times \R$. 
\end{abstract}

\section{Introduction}

In the optimal transportation problem of Monge and Kantorovich,  one is provided with
probability measures $d\mu(x)$ and $d\nu(y)$,  and asked to couple them together so as to minimize
a given transportation cost,  or equivalently to maximize a
given surplus function $s(x,y)$.    The measures  are defined on subsets $\X$ and $\Y$ 
of complete separable metric spaces,  
often Euclidean spaces or manifolds with additional structure,  with the surplus function
$s(x,y)$ either defined by or defining the geometry of the product $\X \times \Y$.
Such problems have a wealth of applications ranging from the pure mathematics
of inequalities, geometry and partial differential equations to topics in 
computer vision,  design, meteorology,  and economics.   These are surveyed in the
books of Rachev and R\"uschendorf \cite{RachevRuschendorf98}, Villani \cite{V} \cite{V2},
Santambrogio \cite{Santambrogio15p} and Galichon \cite{Galichon16p}. 
  It has most frequently been studied
under the assumption that $\X=\Y$,  as in Monge \cite{Monge81}   and Kantorovich \cite{Kant},
or at least that the two spaces $\X$ and $\Y$ have the same finite dimension.  
Monge's question concerned solutions in the form of maps $F:\X \longrightarrow \Y$
carrying $\mu$ onto $\nu$.  At the writing of \cite{V},  Villani described the regularity of such maps 
as the major open problem in the subject.  Through the work of many authors,
an intricate theory has been developed for the case of equal dimensions,  leading up to the restrictive conditions of Ma, Trudinger and Wang \cite{mtw} \cite{TrudingerWang09b} under which the 
solution concentrates on the graph of a {\em smooth} map between $\X$ and $\Y$;  
see \cite{DePhilippisFigalli14} and \cite{Mc14} for complementary surveys.

Despite its relevance to applications,  much less is known when $m\ne n$.   
The purpose of the present paper is to resolve this situation,  at least in the 
case $m>n=1$ motivating our title.  Though it seems not to have received much attention
previously,  we view this as a case interpolating between 
between $m=n=1$,  which can be solved exactly 
in the supermodular case $\p^2 s/\p x \p y>0$,
and the fully general problem $m\ge n>1$.  We develop, for the 
first time, a regularity theory addressing this intermediate case. It is based crucially on a {\em nestedness}
condition introduced simultaneously here and 
in our companion work \cite{ChiapporiMcCannPass15p}.
When satisfied, we show this condition leads
to a unique solution in the form of a Monge mapping.  This solution, moreover is 
semi-explicit: it is based on identifying each level set of the optimal map independently,  
and can presumably be computed with an algorithmic complexity significantly 
smaller than non-nested problems of the same dimensions can be solved.  
Although nestedness is restrictive,
we suspect it may be a requirement for the continuity of optimal maps. Moreover, it
depends subtly on the relation between $\mu,\nu$ and $s$,  which distinguishes it sharply 
from the familiar criteria for mappings, uniqueness and regularity when $m=n$ which, with few exceptions
\cite{KitagawaWarren12}, depend primarily on the geometry and topology of $s$.
Our theory addresses interior regularity,  as well as regularity at some parts of the boundary.
We identify various obstacles to nestedness and regularity along the way,  including
examples which show that higher regularity cannot generally hold on the entire boundary.

The plan for the paper is as follows.  The next section introduces notation, describes
the problem more precisely, and recalls some of its history and related developments.
It is followed by a section dealing with general source and target dimensions $m\ge n \ge1$,  
giving conditions under which the set of $x$ paired with a.e.~$y\in Y$ lie in a codimension $n$
submanifold of $X$.  
In section \S\ref{S:nesting} we specialize to $n=1$,  introduce the notion of nestedness, 
and show that it allows us to obtain a unique solution in the form of a optimal
map between $\mu$ and $\nu$ which, under suitable conditions, is continuous.  The
solution map $F:X \longrightarrow Y$ is constructed one level set at a time.  
Although nestedness generally depends on the relation of $(\mu,\nu)$ to $s$,  
in section \S \ref{S:pseudo-index}
we show that when it happens to hold for all absolutely continuous $(\mu,\nu)$, then $s$ can effectively
be reduced to a super- or submodular function of two real variables $y$ and $I(x)$.
In Section \S \ref{S:level sets}
we explore the motion of the level sets $\f^{-1}(y)$ as $y\in Y$ is varied.  To do so requires additional regularity
of the data $(s,\mu,\nu)$,  under which we are able to deduce some additional regularity of the solution
and give several conditions equivalent to nestedness.  In Section \S\ref{S:examples} we give 
examples of nested and non-nested problems,  including some which illustrate why the unequal dimensions
of the problem prevent $\f$ from extending smoothly to the entire boundary of $X$.  Finally,  in 
Section~\ref{S:regularity},  we develop a complete theory which describes how the higher regularity of $\f$ 
(and the Kantorovich dual potentials) is controlled by certain parameters governing $(s,\mu,\nu,X,Y)$ and various geometrical aspects of the problem, such as the smallness of $\f^{-1}(y)$,  its transversality to $\partial X$,  
and the speed of its motion, which quantifies the uniformity of nestedness by determining 
the separation between distinct level sets.




Our interest in this problem was initially motivated by economic matching problems
with transferable utility \cite{ChiapporiMcCannPass15p}, such as the stable marriage problem  \cite{Knuth97},
in which $\mu$ represents the distribution of female and $\nu$ the distribution of male types,
and $s(x,y)$ represents the marital surplus obtained,  to be divided competitively between a husband of type
$y$ and a wife of type $x$.  The equivalence of optimal transportation to stable matching with transferable
utility was shown in the discrete setting by Shapley and Shubik \cite{SS}.
We adopt this terminology hereafter,  in spite of the fact that $\mu$ 
and $\nu$ might equally well represent producer and consumer locations,  buyer and seller preferences,
etc.   There is no reason a priori to expect the characteristics describing wives and husbands to have the same dimension.  Although we anticipate that this theory will have many other applications,  we are particularly 
aware of its potential for use in the semigeostrophic model of atmosphere and ocean
dynamics,  in which $\mu$ would represent the distribution of fluid in the physical domain and $\nu$ a potential vorticity sheet or filament in dual coordinates~\cite{Cullen06}.

\section{Setting and background results}
\label{S:background}

Given Borel probability measures $\mu$ on $\X \subset \R^m$
and $\nu$ on $\Y \subset \R^n$,  the Monge-Kantorovich problem is to transport
$\mu$ onto $\nu$ so as to optimize the given surplus function $s(x,y)$.   
From Theorem \ref{T:nested} onwards,  we assume $X$ and $Y$ are open,
but for the moment they remain arbitrary.
Assuming 
$s \in C(\X \times \Y)$ to be bounded and continuous
for simplicity,  we seek a Borel measure 
$\gamma \ge 0 $ on $\X \times \Y$ having $\mu$ and $\nu$ for its marginals.  The set 
$\Gamma(\mu,\nu)$ of such $\gamma$ is convex, and weak-$*$ compact
in the Banach space dual to $(C(\X \times \Y),\|\cdot\|_\infty)$.  
Among such $\gamma$,  Kantorovich's problem is to maximize the linear functional
\begin{equation}\label{MK}
MK^* := \max_{\gamma \in \Gamma(\mu,\nu)} \int_{\X \times \Y} s(x,y) d\gamma(x,y).
\end{equation}
We shall also be interested in the structure of the optimizer(s) $\gamma$.
For example,  is there a map $F:\X \longrightarrow \Y$ such that $\gamma$
vanishes outside $Graph(F)$,  and if so,  what can be said about its analytical and
geometric properties?  Such a map is called a {\em Monge} or {\em pure} solution, {\em deterministic coupling},
{\em matching function} or {\em optimal map},
and we have $\gamma = (id \times F)_\#\mu$ where $id$ denotes the identity map on $\X$ in that case \cite{akm}. More generally,
if $F:\X \longrightarrow \Y$ is any $\mu$-measurable map, we define the push-forward $F_\#\mu$ 
of $\mu$ through $F$ by
$$
F_\#\mu(V) = \mu[F^{-1}(V)]
$$
for each Borel $V \subset \Y$.  Thus 
$\Gamma(\mu,\nu) = \{\gamma \ge 0 {\rm\ on}\ \X \times \Y \mid \pi^{X}_\#\gamma = \mu\ \mbox{\rm and}\ 
\pi^Y_\#\gamma =\nu\}$,
where $\pi^X(x,y)=x$ and $\pi^Y(x,y)=y$.  We denote by $\oX$ the closure of $X$, and by
$\spt \mu \subset \oX$ 
the  smallest closed set carrying the full mass of $\mu$.

The case which is best understood is the case $\X=\Y =\R$, so that $n=m=1$.
If $s\in C^2(\R^2)$ satisfies 
\begin{equation}\label{Lorentz}
\frac{\p^2 s}{\p x \p y} (x,y) > 0
\end{equation}
for all $x,y \in \R$, then \eqref{MK} has a unique solution;  moreover,
this solution coincides with the unique measure $\gamma \in \Gamma(\mu,\nu)$ 
having non-decreasing support, meaning $(x,y),(x',y') \in \spt \gamma$ implies 
$(x-x')(y-y') \ge 0$.  This result, which dates back to Lorentz \cite{Lorentz53},
was rediscovered by the economists Becker \cite{Becker73}, Mirrlees \cite{Mirrlees71}
and Spence \cite{Spence73}. Condition \eqref{Lorentz} is also called the Spence-Mirrlees
condition,  or {\em supermodularity},  in the economics literature.  Note this condition does not depend
on $\mu$ or $\nu$,  and while $\gamma$ depends on them,  it is independent of $s$ in this case.
Moreover,  if $\mu$ is  free from atoms,  then $\gamma$ concentrates on the graph of a non-decreasing 
function $F:\R \longrightarrow \R$ given by 
\begin{equation}\label{non-decreasing map}
\int_{(-\infty,F(x))} d\nu \le \int_{-\infty}^x d\mu \le \int_{(-\infty,F(x)]} d\nu.
\end{equation}
When $\mu$ and $\nu$ are given by $L^1$ probability densities $f$ and $g$,  the fundamental theorem of calculus yields
a ordinary differential equation
\begin{equation}\label{1D Monge-Ampere}
f(x)= F'(x) g(F(x))
\end{equation}
satisfied Lebesgue almost everywhere;  smoothness properties of $F(x)$ can then be deduced 
from those of $f$ and $g$.

In higher equal dimensions $m=n>1$ the situation is much more subtle.  However if $\mu$ 
is given by a probability density $f \in L^1(\R^{m})$,  it is again possible to given
conditions on the surplus function $s \in C^1(\X \times \Y)$ such that the Kantorovich
optimizer \eqref{MK} is unique \cite{McCannRifford15p} \cite{cmn} and
concentrated on the graph of a map $F$ \cite{G} \cite{lev} depending
sensitively on the choice of surplus. If $\nu$ is  also given 
by a probability density $g\in L^1(\R^n)$, and $\log f,\log g \in L^\infty$,
then it is possible to give conditions on $s \in C^4(\X \times \Y)$ which guarantee
$F$ is H\"older continuous \cite{loeper} \cite{FigalliKimMcCannARMA13}, and inherits
higher regularity from that of $f$ and $g$ \cite{LiuTrudingerWang10}.

For unequal dimensions $m \ge n$,  the existence, uniqueness and graphical structure of solutions follows from conditions on $s$ as when $m=n$, but concerning other aspects of the problem much less is known.  Only a result of  Pass 
asserts that if smoothness of $F$ holds for all $\log f,\log g\in C^\infty$,  
then the dimensions are effectively equal,  in the sense that there are functions
$I:\R^m \to \R^n$, $\alpha \in C(\R^m)$ and $\sigma\in C(\R ^{2n})$ such that $s(x,y) = \sigma(I(x),y) + \alpha(x)$.
If $n=1$ we call $I$ an  {\em index} and $s$ {\em pseudo-index} in this case;
otherwise $I$ represents a set of indices and $s$ is {\em pseudo-indicial}.


In general,  one of the keys to understanding the Kantorovich problem \eqref{MK}
is the dual linear program
\begin{equation}  \label{dual}
MK_* := \inf_{(u,v) \in Lip_s} \int_{\X} u(x)d\mu(x) +\int_{\Y}v(y)d\nu(y), 
\end{equation}
where $Lip_s$ consists of all pairs of payoff functions $(u,v) \in L^1(\mu) \oplus L^1(\nu)$
satisfying the stability constraint
\begin{equation}
u(x)+v(y)-s(x,y)\geq 0  \label{stability}
\end{equation}
on $\X \times \Y$.  The remarkable fact is that $MK_*=MK^*$ \cite{Kant}.  
Thus,  if $\gamma$ and $(u,v)$ optimize their
respective problems,  it follows that $\gamma$ vanishes outside the zero set $S$ of the 
non-negative function $u+v-s$.
We therefore obtain the first and second order conditions 
\begin{equation}  \label{foc}
(Du(x),Dv(y)) = (D_x s(x,y), D_y s(x,y))
\end{equation}
and 
\begin{equation}  \label{soc}
\left( 
\begin{matrix}
D^2u(x) & 0 \\ 
0 & D^2 v(y)%
\end{matrix}
\right) \ge \left( 
\begin{matrix}
D^2_{xx} s(x,y) & D^2_{xy} s(x,y) \\ 
D^2_{yx} s(x,y) & D^2_{yy} s(x,y)%
\end{matrix}
\right)
\end{equation}at each $(x,y) \in S \cap (\X \times \Y)^0$ for which the
derivatives in question exist; here $\X^0$ denotes the interior of $\X$.
When the surplus $s$ is Lipschitz, 
then the infimum \eqref{dual} is attained by a pair of Lipschitz functions $(u,v)$;
when $s$ has Lipschitz derivatives,  we may even take $u,v$ to be semiconvex,
hence twice differentiable (in the sense of having a second-order Taylor expansion)
Lebesgue almost everywhere \cite{V} \cite{Santambrogio15p}.
We let $\Dom Dv$ and $\Dom D^2 v$ 
denote the domains where $v$ admits a first- and second-order Taylor expansion,
with $\Domo Dv = \left(\overline \Y\right)^0 \cap \Dom Dv$ and 
$\Domo D^2v := \left(\overline \Y\right)^0 \cap \Dom D^2 v$
and $\overline \Y$ denoting the closure of $\Y$.

For each integer $r \ge 0$,  and H\"older exponent $0 < \alpha \le 1$ 
we denote by $C^{r,\alpha}(X)$
the space of functions which are $r$ times continuously differentiable,
and whose $r$-th derivatives are all Lipschitz continuous functions with respect to the 
distance function $|x-x'|^\alpha$ on $X$
(in which case both properties extend to the closure $\overline X$ of $X$.)  We norm this space by
$$
\|f\|_{C^{r,\alpha}(X)} :=  \sum_{i=0}^r \sum_{|\beta| = i} \| D^\beta f \|_{\infty}
+ \sup_{x \ne x' \in X} \sum_{|\beta|=r} \frac{|D^\beta f(x') - D^\beta f(x)|}{|x'-x|^\alpha} 
$$
where $D^\beta f= \frac{\p^{|i|}f}{\p x_1 \cdots \p x_i}$ and the sums are over multi-indices $\beta$
of degree $|\beta|$. We eliminate the supremum if $\alpha=0$, abbreviating $C^r(X) := C^{r,0}(X)$, $C(X) := C^0(X)$, and defining local
versions $C^{r,\alpha}_{loc}(X)$ of these spaces analogously.
We adopt the convention that $C^{-1,1}=L^\infty$,  and denote by $\|\p X\|_{C^{r,1}(X')}$ a minimal bound
for the sums of norms of the $C^{r,1}$ functions parameterizing $\overline{X'} \cap \p X$.
In case $\alpha=1$,  this space forms an algebra; moreover,
we extend these definitions to tensor fields $f$ on Riemannian manifolds $X$ 
by requiring the components of $f$ to belong to $C^{r,1}_{loc}$ and setting 
$$
\|f\|_{C^{r,1}(X)} :=  \sum_{i=0}^{r+1} \| D^i f \|_{\infty},
$$
the operator $D^i$ now denoting iterated covariant derivatives with respect to the Levi-Civita connection.

We shall also use the Sobolev spaces $W^{r,1}(X)$,  consisting of integrable functions
on $X$ whose distributional partial derivatives up to order $r$ are also integrable. This space is normed by
$$
\|f\|_{W^{r,1}(X)} := \int_X ( \sum_{|\beta|\le r} |D^\beta f| ) d{\mathcal H^m},
$$
where the sum is over multi-indices $\beta$ and 
${\mathcal H^m}$ denotes the Hausdorff $m$-dimensional measure on $X$.
Equipped with the norm 
$$\|f\|_{(C \cap W^{1,1})(X)} := \max\{\|f\|_{L^\infty(X)},\| f \|_{W^{1,1}(X)}\},
$$
the space $(C \cap W^{1,1})(X)$ forms an algebra (closed under the continuous operation
of multiplication),  as does $C(Y; (C \cap W^{1,1})(X))$.

\section{Transportation between unequal dimensions}
\label{S:unequal dimensions}

We now turn our attention to the case in which the source and target domains $\X$ and $\Y$
of our transportation problem have unequal dimensions $m\geq n$.  In this case,
one expects many-to-one rather than one-to-one matching. Indeed, it is
natural to expect that at equilibrium the subset $\f^{-1}(y)\subset \X \subset \R^m$ of
partners which a man of type $y\in \Dom_0 Dv$ is indifferent to
will generically have dimension $m-n$, or equivalently, codimension $n$. 
Our first proposition recalls from \cite{ChiapporiMcCannPass15p} 
conditions under which this indifference set
will in fact be a Lipschitz (or smoother) submanifold of the expected dimension.
It is stated here for costs $s \not\in C^2$ which need not be quite as smooth as those 
of  \cite{ChiapporiMcCannPass15p}.

\subsection{Potential indifference sets}

For any optimal $\gamma$ and payoffs $(u,v)$, we have already
seen that $(x,y) \in S \cap (\X\times \Domo Dv)$ implies 
\begin{equation}\label{yfoc}
D_y s(x,y) = Dv(y).
\end{equation}
That is, all partner types $x \in \X$ for husband $y \in \Domo Dv$ lie in the same level set of the map $x \mapsto D_y s(x,y)$. If
we know $Dv(y)$, we can determine this level set precisely; it depends on $%
\mu$ and $\nu$ as well as $s$. However, in the absence of this knowledge it
is useful to define the \emph{potential indifference sets}, which for given 
$y \in \Y$ are merely the level sets of the map $x\in \X \mapsto D_y s(x,y)$.
We can parameterize these level sets by (cotangent) vectors 
$k \in T_y^*\Y =\R^n$: 
\begin{equation}  \label{cotangent-parameterization}
X(y,k) := \{ x \in \X \mid D_y s(x,y) = k\},
\end{equation}
or we can think of $y \in \Y$ as inducing an equivalence relation between
points of $\X$, under which $x$ and $\bar x \in \X$ are equivalent if and only
if 
\begin{equation*}
D_y s(x,y) = D_y s(\bar x,y).
\end{equation*}%
%
%
%
%
%
%
%
%
Under this equivalence relation, the equivalent classes take the form %
\eqref{cotangent-parameterization}. We call these equivalence classes \emph{%
potential indifference sets}, since they represent a set of partner types
which $y \in \Domo Dv$ has the potential to be indifferent
between. 

A key observation concerning potential indifference sets is the following proposition.
{Recall for a Lipschitz function $F:\R^m \longrightarrow \R^n$,  the {\em generalized 
Jacobian} or {\em Clarke subdifferential} $\p F(x)$ at $x \in \R^m$ consists of the convex hull of limits of the
derivatives of $F$ at nearby points of differentiability \cite{Clarke83}.  For example, $\p F(x) = \{DF(x)\}$ if $F$ is $C^1$ at $x$.}


\begin{definition}[Surplus degeneracy]
Given $\X \subset \R^m$ and $\Y \subset \R^n$ with $m\ge n$, we say $s \in C^2(\X \times \Y)$ 
{\em degenerates} at $(\bar x,\bar y) \in X \times Y$ if $rank(D^2_{xy}s(\bar x,\bar y))<n$.
{If $s \in C^{0,1}_{loc}(X\times Y)$ and $D_y s$  is locally Lipschitz,  we say $s$
degenerates at $(\bar x,\bar y) \in X \times Y$ if for every Lipschitz
extension of $D_ys$ to a neighbourhood of $(\bar x,\bar y)$ and choice of orthonormal basis for $\R^m$,
setting 
\begin{equation}\label{IFT coordinates}
F(x) = D_y s(x,\bar y)
\end{equation}
yields some $m \times n$ matrix $M \in \partial F(\bar x)$ with $\det [M_{ij}]_{1\le i,j\le n} = 0$.} 
Otherwise we say $s$ is {\em non-degenerate} at $(\bar x,\bar y)$.
\end{definition}


\begin{proposition}[Structure of potential indifference sets]
\label{P:indifference structure} Let $s,D_ys \in C^{r,1}_{loc}(\X \times \Y)$ for some $%
r\ge 0$, where $\X \subset {\mathbf{R}}^m$ and $\Y \subset {\mathbf{R}}^n$
with $m \ge n$.
If  $s$ does not degenerate at $(\bar x,\bar y) \in \X
\times \Y$, then $\bar x$ admits a neighbourhood $U \subset {\mathbf{R}}^m$
such that $X(\bar y,D_y s(\bar x,\bar y)) \cap U$ coincides with the intersection of $\X$
with a $C^{r,1}$-smooth, codimension $n$ submanifold of $U$.
\end{proposition}

\begin{proof}
For smooth $s$,  the set $L:=\{ x \in U \mid D_y s(x,\bar y) = D_y s(\bar x,\bar y) \}$ forms
a codimension $n$ submanifold of $U$, by the preimage theorem \cite[\S 1.4]%
{GuilleminPollack74}.   Otherwise choose an $C^{r,1}$ extension of $D_ys$ to a neighbourhood $U \times V$ of 
$(\bar x,\bar y)$ and an orthonormal basis for $\R^m$ such that each $M \in \p F$ 
has full rank with $F$ from \eqref{IFT coordinates}.
The Clarke inverse function theorem \cite{Clarke83} gives a $C^{r,1}$ local inverse
to the map $x\in U \mapsto (D_y s(x,\bar y),x_{n+1},\ldots,x_m)$.  Taking $U$ smaller if necessary, the set 
$X(\bar y,D_ys(\bar x,\bar y))$ is the 
image of the 
affine subspace 
$\{D_y s(\bar x,\bar y)\} \times \R^{m-n}$ 
under this  (biLipschitz) local inverse.
\end{proof}

Although we have stated the proposition in local form, it implies that if $%
\bar k = D_y s(\bar x,\bar y)$ is a \emph{regular value} of $x \in \X \mapsto
D_y s( x,\bar y)$ --- meaning $s$ is non-degenerate throughout 
$X(\bar y,\bar k)$ --- then $X(\bar y, \bar k)$ is
the intersection of $\X$ with an $m-n$ dimensional submanifold of ${\mathbf{R}%
}^m$.

\begin{remark}[Genericity]
For $s \in C^2$ regular values are generic in the sense that for any given $\bar{y}\in \overline\Y$,
Sard's theorem asserts that the regular values of $D_{y}s(\cdot ,\bar{y})$
form a set having full Lebesgue measure in ${\mathbf{R}}^{n}$. However, if $%
D_{y}s(\cdot ,\bar{y})$ also has \emph{critical} (i.e.\ non-regular) values
for some $\bar{y}\in \overline Y$, it is entirely possible that $Dv(\bar{y})$ is a {%
critical} value of $D_{y}s(\cdot ,\bar{y})$ for each such $\bar{y}$. This is
necessarily the case when $rank(D_{xy}^{2}s(x,y))<\min \{m,n\}$ throughout $%
X\times Y$, meaning $s$ is globally degenerate.
\end{remark}


As argued above, the potential indifference sets %
\eqref{cotangent-parameterization} 
are determined by the surplus function $s(x,y)$ without reference to the
populations $\mu $ and $\nu $ to be matched. On the other hand, the
indifference set actually realized by each $y\in Y$ depends on the
relationship between $\mu $, $\nu $ and $s$. This dependency is generally
complicated. However, there is one case in which it may simplify
substantially: the case of multi-to-one dimensional matching, namely $n=1$.
In this case, suppose $D_{xy}^{2}s(\cdot ,y)$ is non-vanishing (i.e.\ $\frac{%
\partial s}{\partial y}(\cdot ,y)$ takes only regular values). 
Then the potential indifference sets $X(y,k)$ form hypersurfaces in $\R^m$.
Moreover, as $k$ moves through ${\mathbf{R}}$, these potential indifference
sets sweep out more and more of the mass of $\mu$. For each $y\in Y$ there
will be some choice of $k\in {\mathbf{R}}$ for which the $\mu $ measure of $%
\{x\mid D_{y}s(x,y) \le k\}$ exactly coincides with the $\nu $ measure of $%
(-\infty ,y]$ (assuming both measures are absolutely continuous with respect
to Lebesgue, or at least that $\mu $ concentrates no mass on hypersurfaces
and $\nu $ has no atoms). In this case the potential indifference set $%
X(y,k) $ is said to split the population proportionately at $y$, 
making it a natural
candidate for being the {\em iso-husband} set $\f^{-1}(y)$ to be matched with 
$y$.\footnote{Since $k=s_y(x,y)$ can be recovered from any $x\in X(y,k)$ and $y$,  we may
equivalently say $x$ splits the population proportionately at $y$,
and vice versa.} 
In the next sections, we go on to describe and contrast situations in
which this expectation is born out and leads to a complete solution from
those in which it does not.

\section{Multi-to-one dimensional matching}
\label{S:nesting}

We now detail a new approach to a specific class of transportation problems, largely
unexplored, but which can
often be solved explicitly as below. These are 
\emph{multi-to-one dimensional} problems, 
in which the space of wives may have several dimensions but the space of
husbands only one.  Thus, we are matching a 
distribution on $x=\left( x_{1},...,x_{m}\right) \in {\mathbf{R}}%
^{m}$ with another on $y\in {\mathbf{R}}$. The surplus $s$ is then a
function $s\left( x_{1},...,x_{m},y\right) $ of $m+1$ real variables. 

The goal is to construct
from data $(s, \mu,\nu)$ a mapping $\f:\X \longrightarrow \Y \subset \R$,
whose level sets $\f^{-1}(y)$ constitute the {\em iso-husband} sets, or submanifold of wives among which
husband $x$ turns out to be indifferent facing the given market conditions.
At the end of the preceding section we identified a natural candidate for this
iso-husband set: namely the potential indifference set which divides the mass 
of $\mu$ in the same ratio as $y$ divides $\nu$;  whether or not these natural
candidates actually fit together to form the level sets of a function $F$ or not
depends on a subtle interaction between $\mu$, $\nu$ and $s$.  When
they do,  we say the model is {\em nested},  and in that case we show
that the resulting function $\f:X \longrightarrow Y$ produces the unique 
optimizer $\gamma=(id \times \f)_\#\mu$ for \eqref{MK}.
Note that except in the Lorentz/Becker/Mirrlees/Spence case $m=1=n$,  
this nestedness depends not only on $s$, but also on $\mu$ and $\nu$.

\subsection{Constructing explicit solutions for nested data}

For each fixed $y\in Y\subseteq {\mathbf{R}}$, our goal is to identify the 
iso-husband set $\{x\in X\mid \f(x)=y\}$ of husband type $y$
in the given problem. When differentiability of $v$ holds at $%
y\in \Y^0$, the argument in the preceding section implies that this is contained in
one of the potential indifference sets $X(y,k)$ from \eqref{cotangent-parameterization}.
Proposition \ref{P:indifference structure} indicates when this set will have
codimension $1$; it generally divides $\X$ into two pieces: the sublevel set
\begin{eqnarray}\label{Xsubseteq}
X_\le(y,k) &:=& \{x \in  \X \mid \frac{\partial s}{\partial y} 
(x,y) \le k \},
\end{eqnarray} 
and its complement $X_>(y,k):= \X \setminus X_\le (y,k)$.
We denote its strict variant by $X_<(y,k) := X_\le(y,k) \setminus X(y,k)$.

To select the appropriate level set,  we
choose the unique level set \textit{splitting the population proportionately}
with $y$; that is, the $k=k(y)$ for which the $\mu$ measure of female types 
$X_\le(y,k)$
coincides with the $\nu$
measure of male types $(-\infty,y]$. We then set $y:=\f(x)$ for each $x$ in
$X(y,k)$. 
Our next theorem specifies conditions under which the resulting match
$\gamma = (id \times \f)_\#\mu$ optimizes the Kantorovich problem \eqref{MK};
we view it as the natural generalization of the positive assortative matching
results of \citeasnoun{Lorentz53} \citeasnoun{Mirrlees71} \citeasnoun{Becker73} and \citeasnoun{Spence73}  
from the one-dimensional to the multi-to-one dimensional setting.  
Unlike their criterion \eqref{Lorentz}, which depends only on $s$,
ours relates $s$ to $\mu$ and $\nu$,  by
requiring the sublevel sets $y \in Y \mapsto X_\le(y,k(y))$ 
identified by the procedure above to depend monotonically on $y\in \R$,
with the strict inclusion $X_\le(y,k(y))\subset X_<(y',k(y'))$ holding whenever $\nu[(y,y')]>0$.
We say the model $(s,\mu,\nu)$ is {\em nested} in this case.




\begin{theorem}[Optimality of nested matchings]
\label{T:nested}
Let $X \subset \R^m$ and $Y  \subset \R$ be connected open sets equipped with Borel probability measures
$\mu$ and $\nu$.  Assume $\nu$ has no atoms and $\mu$ vanishes on each Lipschitz hypersurface.
Use $s \in C^{1,1}(X\times Y)$ and $s_y =\frac{\partial s}{\partial y}$ 
to define $X_\le$, $X_<$, 
etc. as in \eqref{Xsubseteq}.


(a) Assume  $s$ is non-degenerate 
throughout ${X \times Y}$.
Then for each $y \in \Y$ there is a maximal interval $K(y)=[k^-(y),k^+(y)] \ne \emptyset$ such that
$\mu[X_\le(y,k)] =\nu[(-\infty,y)]$ for all $k\in K(y)$.  Both $k^+$ and $(-k^-)$ are upper semicontinuous.  

(b) In addition, assume both maps $y \in Y \mapsto X_\le (y,k^\pm(y))$ are non-decreasing,  and moreover that
$\int_y^{y'} d\nu >0$ implies $X_\le (y,k^+(y)) \subseteq X_< (y',k^-(y'))$.
Then 
$k^+$ is right continuous, $k^-$ is left continuous, and they agree throughout
$\spt \nu$ except perhaps at countably many points.
Setting $\f(x)=y$ for each $x \in X(y,k^+(y))$ defines $\f:X \longrightarrow Y$ [$\mu$-a.e.], 
 and $\gamma=(id \times \f)_\#\mu$ is the unique maximizer
of \eqref{MK} on $\Gamma(\mu,\nu)$.  

(c) In addition, assume $\spt \nu$ is connected, say $\spt \nu = [\underline y, \bar y]$. 
Then $\f$ agrees $\mu$-a.e.\ with the continuous function
\begin{equation}\label{bar F continuous}
\bar \f(x) = \left\{
\begin{array}{ccl}
\bar y & & {\rm if}\ x \not \in 
\bigcup_{y \in (\underline{y},\bar y)} X_<(y,k^-(y))
\\ y  & &  {\rm if}\   x \in X_\le(y,k^+(y)) \setminus X_<(y,k^-(y))\ \mbox{\rm with}\ y \in (\underline y,\bar y),
\\ \underline y & & {\rm if}\ x \in 
\bigcap_{y \in (\underline{y},\bar y)} X_\le(y,k^+(y)).
\end{array}
\right.
\end{equation}
(d) If, in addition, $Y \subset \spt \nu$ then $\bar \f:X \longrightarrow Y$.
\end{theorem}

Our strategy for proving (b) is to show the $\spt \gamma$ is contained in the 
the {\em $s^*$-subdifferential}
$$
\partial^{s^*}v^* := \{(x,y) \in \oX \times \oY \mid v(y') \ge v(y) - s(x,y) + s(x,y') \quad \forall y' \in \oY \}.
$$
of the Lipschitz function $v:\oY \longrightarrow \R \cup \{+\infty\}$ solving the differential equation $v'(y)= k^+(y)$ a.e.
From there, we can conclude optimality of $\gamma$ using well-known results 
which show $s^*$-subdifferentials have a property (known as {\em $s$-cyclical monotonicity})
characterizing the support of optimizers \cite{GM} \cite{V2}.  We also
give a self-contained duality-based proof of optimality of $\gamma$ 
as a byproduct of our uniqueness argument.

\begin{proof}
(a) Proposition \ref{P:indifference structure} implies 
$X(y,k)$ 
is the intersection with $\X$ of an $m-1$ dimensional Lipschitz submanifold (orthogonal to 
$D_x s_y(x,y) \ne 0$ wherever the latter is defined and continuous).
Since both $\mu$ and $\nu$ vanish on hypersurfaces, the 
function
$$
h(y,k):=\mu[X_\le(y,k)]-\nu[(-\infty,y)] 
$$
is continuous, and for each $y \in Y$ climbs monotonically from $-\nu[(-\infty,y)]$ to $1- \nu[(-\infty,y)]$
with $k \in \R$. The intermediate value theorem then implies the existence of $k^\pm(y)$.  Continuity of $h(y,k)$ also confirms that the zero set $[k^-(y),k^+(y)]$ of $k\mapsto h(y,k)$ is closed, that $k^-$ is lower semicontinuous and $k^+$ is upper semicontinuous.

(b) Observe $k^-(y)<k^+(y)$ implies the open set $X_<(y,k^+(y))\setminus X_\le(y,k^-(y))$ is non-empty,
because the image of $x \in X \mapsto s_y(x,y)$ is connected.  This open subset of $X$ is disjoint from
$X_<(y',k^+(y'))\setminus X_\le(y',k^-(y'))$ whenever $\nu[(y,y')] \ne 0$,  by the monotonicity assumed of
$X_\le(\cdot,k^\pm(\cdot))$.
Since $X$ can only contain countably many disjoint open sets,
we conclude $k^+(y)=k^-(y)$ for $y \in \spt \nu$,  
apart perhaps from countably many points.


Having established $\nu[\{y \mid k^+>k^-\}]=0$,
we shall continue the proof by showing the solution
$$
v(y) := \int^y k^+(y) dy.
$$
of $v'(y) = k^+(y)$ has a $s^*$-subdifferential which is closed and contains $S^+$ defined by
$$S^\pm:=\{(x,y) \in 
X \times Y 
\mid s_y(x,y) =k^\pm(y)\}.$$ 
(A similar argument shows the solution to $v'= k^-$ has a $s^*$-subdifferential containing $S^-$.) 
Note that the global bound $s_y \in L^\infty$ implies $v$ is Lipschitz, hence differentiable
Lebesgue almost everywhere.

To begin,  suppose $(x',y') \in S^+$ and $k^+(y')=k^-(y')$,  so that $v$ is differentiable at $y'$
and $v'(y') = s_y(x',y')$.  Since
$X_\le(y',v'(y')) \subseteq X_\le(y'+\delta,v'(y'+\delta))$ for almost all $\delta>0$,
we see $s_y(x',y) \le v'(y)$ for almost all $y>y'$.  Integrating over $y\in[y',y^+]$ yields 
\begin{equation}\label{right b-subdifferentiability}
v(y^+) -s(x',y^+) \ge v(y') - s(x',y')
\end{equation}
for all $y^+ \ge y'$. On the other hand, since $v'(y)$ is continuous at $y'$ and 
$X(y,v'(y))$ is a hypersurface in the open set $X$,  
we deduce $x' = \lim_{i \to \infty} x_i$ for some sequence satisfying  
$s_y(x_i,y')>v'(y')$. Then from $X_\le(y-\delta,v'(y-\delta)) \subset X_\le(y,v'(y))$
we deduce $s_y(x_i, y'-\delta) > v'(y'-\delta)$ for almost all $\delta>0$
hence $s_y(x',y'-\delta) \ge v'(y'-\delta)$.  Integrating
over $y-\delta \in [y^-,y']$ yields
\begin{equation}\label{left b-subdifferentiability}
s(x',y')-v(y') \ge s(x',y^-) - v(y^-)
\end{equation}
for all $y^- \le y'$.  
Together \eqref{right b-subdifferentiability}--\eqref{left b-subdifferentiability}
show $(x',y') \in \partial^{s^*} v^*$.

It remains to consider the countable collection of points $y\in \spt \nu$ where $k^+(y)>k^-(y)$.
We claim $k^+$ is right continuous. 
Since the $s$-subdifferential $\partial^{s^*}v^*$ is closed,  this will complete the proof that
$S^+ \subset \partial^{s^*}v^*$.
Due to the semicontinuity
already established,  it is enough to show
$$
k^+(y) \le L^+:=\lim \inf_{\delta \downarrow 0} k^+(y+\delta).
$$

For any sequence $y_i \ge y_{i+1}$ decreasing to $y$ with $k^+(y_i) \to L^+$,
the sets $X_\le(y_i,k^+(y_i))$ decrease by assumption.   The limit set 
$X_\infty := \bigcap_{i=1}^\infty X_\le (y_i,k^+(y_i))$ therefore contains $X(y,k^+(y)) \subseteq X_\infty$.  
On the other hand, from the definition and continuity of $s_y$ we check 
$X_\infty \subseteq X_\le(y,L^+)$. Monotonicity of $k \mapsto X_\le(y,k)$ yields $k^+(y) \le L^+$ as 
desired.  The proof that $k^-$ is left continuous is similar, and implies $s$-cyclical monotonicity of $S^-$.



Since $k^\pm$ are both constant on any interval $I\subset \R$ with
 $\nu[I]=0$,  we conclude they are continuous functions except perhaps
at the countably many points in $\spt \nu$ where they disagree  ($k^-$ being left
continuous and $k^+$ right continuous at such points).  
We shall show $\gamma:=(id \times \f)_\# \mu$ is well-defined and maximizes \eqref{MK}.

The definitions of $\f$ and $S^+$ were chosen to ensure $Graph(\f) \subset S^+$.
There are two kinds of irregularities to consider.
Any discontinuities in $k^+$ correspond to countably many gaps 
\begin{equation}\label{open gaps}
X_<(y,k^+(y)) \setminus X_< (y,k^-(y))
\end{equation}
in the distribution $\mu$ of wives.  Although $F$ has not been defined on these open gaps,
they have $\mu$ measure zero.  In addition, $Y \setminus \spt \nu$ may consist of at most 
countably many intervals $(\underline y_i, \bar y_i)$.  To each such interval corresponds
a relatively closed gap
\begin{equation}\label{closed gaps}
X_\le(\bar y_i,k^+(\bar y_i) \setminus X_<(\underline y_i,k^-(\underline y_i))
\end{equation}
of $\mu$ measure zero on which $F$ is not generally well-defined;  typically this gap
consists of a single Lipschitz submanifold,  throughout which we have attempted to simultaneously
assign  $F$ every value in $[\underline y_i, \bar y_i]$.


Apart from these two kinds of gaps (open and closed) --- which have $\mu$ measure zero and on which $F$
may not be (well-)defined ---
the $m-1$ dimensional Lipschitz submanifolds $X(y,k^+(y))$ foliate the support of $\mu$ in $X$.  The leaves
of this foliation correspond to distinct values $y\ne y'$ and are disjoint since $\nu([y,y'])>0$
implies $X(y,k^+(y)) \subset X_<(y',k^+(y'))$ disjoint from $X(y',k^+(y'))$.
This shows $\f$ is defined $\mu$-a.e.;  it is also $\mu$-measurable since $\f^{-1}((-\infty,y])$
differs from a relatively closed subset of $X$ only by the countably many above-mentioned gaps, 
which are themselves Borel and have vanishing $\mu$ measure.
Since $\f$ was selected so that 
$\f_\# \mu := \mu \circ \f^{-1}$ agrees with $\nu$ on the subalgebra of intervals $(\infty,y]$
parameterized  by $y\in Y$,  we see 
$\gamma := (id \times \f)_\#\mu$ lies in $\Gamma(\mu,\nu)$ and is well-defined.  Since 
$Graph(\f) \subset S^+ \subset \partial^{s^*} v^*$, we conclude
 $\gamma$ vanishes outside $\partial^{s^*} v^*$.  On the other hand,
$s^*$-subdifferentials are known to be $s$-cyclically monotone \cite{Rochet87},
whence optimality of $\gamma$ follows from standard results \cite{V2}.  We later give
a self-contained proof, as a byproduct of our uniqueness argument below.

(c) Before addressing uniqueness, let us consider point (c).  
Each point $x\in X$ outside the $\mu$-negligible gaps \eqref{open gaps}--\eqref{closed gaps}
belongs to a leaf $X(y,k^+(y))$ of the foliation corresponding to a unique $y\in Y$;
moreover $k^+(y)=k^-(y)$.  Since $X(y,k^+(y)) = X_\le(y,k^+(y))\setminus X_<(y,k^-(y))$
for such $y$ we see $\f=\bar \f$ holds $\mu$-a.e.  Let us now argue $\bar \f$ is continuous.

The fact that $\spt \nu = [\underline y,\bar y]$
is connected rules out all gaps \eqref{closed gaps} of the second kind in $X$ excepting two, 
corresponding to the intervals in $Y$ to the left of $\underline y$ and to the right of $\bar y$.
Thus $\bar F$ is well-defined throughout $X$, taking constant values $\underline y$ and $\bar y$
on each of these two gaps,  and the constant value $\bar F(y)=y$ on each remaining open gap
\eqref{open gaps}.  From \eqref{bar F continuous} and the monotonicity assumed of 
$y \in Y \mapsto X_\le(y,k^\pm(y))$ we see
$$
\bar F^{-1}((-\infty,y])=\left\{
\begin{array}{ccc}
\emptyset & {\rm if} &y< \underline y,
\\ \displaystyle \bigcap_{y' \in (\underline{y},\bar y)} X_\le(y',k^+(y'))
& {\rm if} & y=\underline y,
\\ X_\le(y,k^+(y)) & {\rm if} & y \in (\underline y,\bar y),
\\ X & {\rm if} & y \ge \bar y,
\end{array}
\right.
$$ 
and
$$
\bar F^{-1}([y,\infty))=\left\{
\begin{array}{ccc}
X & {\rm if} &y\le \underline y,
\\ X_\ge(y,k^-(y)) & {\rm if} & y \in (\underline y,\bar y),
\\ 
X \setminus \bigcup_{y' \in (\underline{y},\bar y)} X_<(y',k^-(y')) & {\rm if} & y = \bar y,
\\ \emptyset & {\rm if} & y> \bar y,
\end{array}
\right.
$$ 
are relatively closed, hence $\bar F$ is continuous.

(d) is obvious from the monotonicity of $X_\le(y,k^+(y))$.
\end{proof}

\begin{proof}[Proof of  uniqueness in Theorem \ref{T:nested}(b)]
To prove unicity, we show that the function $v$ constructed above,
along with 
\begin{equation}\label{v conjugate}
u(x) := \sup_{y \in Y} s(x,y) - v(y)
\end{equation}
solve the dual problem \eqref{dual}.  
(Since $v$ is Lipschitz, we may equivalently take the supremum defining $u$ over $Y^0$ or $\overline Y$.)
From there we proceed to argue along well-trodden lines.
From \eqref{v conjugate} we see
\eqref{stability} holds on $X \times Y$.
Moreover, the fact that $Graph(F) \subset \partial^{s^*} v^*$ shows for $x \in \Dom F$
that the supremum \eqref{v conjugate} is attained at $y=F(x)$;  thus 
$u<\infty$ on $\Dom F$.  This implies $v$ has a global lower bound,
which in turn implies a global upper bound for $u$.  On the other hand,
$$
v(y) \ge \sup_{x \in X} s(x,y) - u(x)
$$
shows $u$ is bounded below,  hence belongs to $L^\infty$.  Since $y=F(x)$ yields 
equality in the stability constraint \eqref{stability}, integrating $s$ against $\gamma = (id \times F)_\#\mu$ yields
\begin{eqnarray*}
MK^* \ge  \int s d(id \times F)_\#\mu 
&=& \int [u(x) +  v(F(x))] d \mu(x) 
\\ &=& \int u d\mu + \int v d\nu
\\ &\ge& MK_*,
\end{eqnarray*}
where we have concluded $v \in L^1(\nu)$ and hence $(u,v) \in Lip_s$ from the second equality.
Since the opposite inequality $MK^* \le MK_*$ is immediate from \eqref{stability} and $\gamma \in \Gamma(\mu,\nu)$,  we conclude that our functions $(u,v)$ attain the infimum \eqref{dual}.
(We also obtain direct confirmation that $MK^*=MK_*$ and $(id \times \f)_\# \mu$ attains the supremum \eqref{MK}.)

Now let $\bar \gamma$ be any other optimizer for \eqref{MK}.  To establish uniqueness,
we shall show $\bar \gamma$ vanishes outside the graph of $F$,  after which \cite[Lemma 3.1]{akm} 
equates $\bar \gamma$ with $(id \times F)_\#\mu$ to conclude the proof.
The set $S:=\spt \gamma \cap (X \times \Dom_0 Dv)$ carries the full mass of $\bar \gamma$.
Each $(x',y') \in S$ produces equality in \eqref{stability},  hence $s_y(x',y') = v'(y')$.
Thus $x' \in X(y', k(y'))$ and $F(x') =y'$, showing $S \subset Graph(F)$ as desired.
\end{proof}

\subsection{Universally nesting surpluses reduce dimension}
\label{S:pseudo-index}

Nestedness is generally a property of the three-tuple $\left( s,\mu ,\nu
\right) $; that is, for most  non-degenerate surplus functions, the model may or may not be
nested depending on the measures under consideration. In some cases,
however, the surplus function is such that the model is nested for all
measures $\left(\mu ,\nu \right) $.   We show 
this occurs precisely when the surplus takes \emph{pseudo-index} form; for $s\in C^2$ this means
there exist $C^1$ functions $\alpha$ and $I$ on $X\subset \R^m$  and 
$\sigma$ on $I(X) \times Y \subset \R^2$ 
such that
\begin{equation}\label{pseudo-index}
s\left( x,y\right) =\alpha \left( x \right) +\sigma
\left( I\left( x\right) ,y\right).
\end{equation}
In this case the different components of $x=(x_1,\cdots,x_m)$ are relevant only in so far
as the determine the value of the index function $I(x)$, and the dimension of the transport problem is effectively reduced from $m+1$ to $1+1$,  since it becomes equivalent to matching $I_\#\mu$ with $\nu$ optimally for $\sigma$.  For connected domains,  we shall also see non-degeneracy implies the effective surplus function $\sigma$  is either super- or sub-modular.




\begin{proposition}[Surpluses nesting universally  are pseudo-index]\label{P:pseudo-index}
Assume $X \subseteq \mathbf{R}^n$ and $ Y=(a,b) \subset \mathbf{R}$ are open and connected, and that 
the surplus $s \in C^2(X \times Y)$ is non-degenerate, meaning $D_x s_y$ is nowhere
vanishing on $X \times Y$.  Then $s$ takes pseudo-index form \eqref{pseudo-index} if and only if 
$(s, \mu,\nu)$ is nested for every choice of absolutely continuous probability measures $\mu$ on $X$  and
 $\nu$ on $Y$.
\end{proposition}

Under the additional hypothesis that each level set $X(y,k)$ is connected, the "if" assertion in the preceding proposition follows immediately from Theorem 3.2 in \cite{P2} (which, under the additional connectedness assumption, asserts that if the optimal map is continuous for all  marginals, with densities bounded above and below, then $s$ must be of index form) and Theorem \ref{T:nested}: if $(s,\mu,\nu)$ is nested for each $(\mu,\nu)$, Theorem \ref{T:nested} implies that the optimal map is continuous whenever $\nu$ has connected supported, contradicting  the result in \cite{P2}.

We offer a self-contained proof in an appendix.  Aside from eliminating the connectedness assumption on the $X(y,k)$, our proof has the additional advantage of being elementary;  the proof of Theorem 3.2 in \cite{P2} relies on a sophisticated result of Ma, Trudinger and Wang \cite{mtw}, which asserts that $s$-convexity of the target is a necessary requirement  for optimal maps to be continuous for all marginals with densities bounded above and below.

\section{Criteria for nestedness}
\label{S:level sets}


Theorem \ref{T:nested} illustrates the powerful implications of nestedness,
when it is present.  In this section,  we exploit the machinery of level set dynamics
to derive several alternative characterizations of nestedness 
which --- for sufficiently smooth data --- may be easier to check in practice.   
The first of these asserts that nestedness is
essentially equivalent to $X_\le(y,k^+(y))$ expanding outward at each point on its boundary
as $y$ is increased,  assuming its boundary hits $\partial X$ transversally. 
To describe this outward expansion requires us to derive an analytic expression for the normal
velocity of $X_\le(y,k^+(y))$;  this normal velocity also appears naturally in the integrodifferential analog of the 
Monge-Amp\`ere type equation adapted to multi-to-one dimensional transport which is derived 
at \eqref{MA type} below.
A second characterization, also requiring transversality, asserts nestedness is equivalent to the existence
of a unique mapping $F:X \longrightarrow Y$ such that $y=\f(x)$ splits the population proportionately to $x$ for 
each $x \in X$.

We begin with a preparatory lemma. Anticipating our later needs, we allow for the possibility of matching
between the line and a Riemannian manifold $M$ with metric tensor $g_{ij}$ of Sobolev regularity
(such as $M=\p X\subset \R^{m}$ with $\p X \in C^1 \cap W^{2,1}$);
until then the reader may imagine $(M,g_{ij})$ represents Euclidean space.
Hausdorff measure ${\mathcal H}^d$ of dimension $d \le m$ on $M$,
functions of bounded variation,  sets of finite perimeter and their reduced boundaries are defined as in, e.g.~\cite{EvansGariepy92} for $M=\R^n$ and \cite{Federer69} for the general case.



\begin{lemma}[Motion of sublevel sets]\label{L:level set motion}
Let $M$ be a complete $m$-dimensional manifold with a Riemannian metric tensor 
$g_{ij} \in (C \cap W^{1,1})(M)$.
Let $X \subset M$ and $Y\subset \R$ be 
domains of finite perimeter. Fix $s \in C^{0,1}$
non-denegerate throughout $\overline {X} \times \overline Y$,
with $s_y \in C^1(\oX \times \oY)$ and each component of $(D_x s_y,s_{yy})$ in $C(Y; (C\cap W^{1,1})(X))$.
Then
\begin{eqnarray}\label{no surface plateaus}
N := \{ (y,k) \in \overline Y \times \R \mid {\mathcal H}^{m-1}\left[\oX(y,k) \cap \p^* X\right]>0
\}
\end{eqnarray}
is closed and $N \cap  \{y \} \times \R$ is countable for each $y\in\overline Y$,
where 
$\partial^* X \subset \partial X$ denotes the reduced boundary of $X$.  
Setting $U:= \overline Y \times \R$, the area ${\mathcal H}^{m-1}[X(y,k)]$ varies continuously on 
$U \setminus N$.

If $f \in L^\infty(X \times Y)$ and $f_y := \frac{\p f}{\p y} \in L^\infty$, then
$$
\Phi(y,k) := \int_{X_\le(y,k)} f(x,y) d{\mathcal H}^m (x)
$$
defines a Lipschitz function on 
$U:=\overline Y \times \R$, with $\|\Phi\|_{C^{0,1}(U)}$ controlled
by the volume and perimeter 
of $X$, $\|(f,f_y,s_{yy},|D_x s_y|^{-1})\|_{L^\infty(X\times Y)}$ and 
\begin{equation}\label{level set upper bound}
\sup_{(y,k) \in U} {\mathcal H}^{m-1}[X(y,k)] <\infty.
\end{equation}
For a.e.\ $(y,k) \in U$,
\begin{eqnarray}\label{DPhi}
D\Phi(y,k) 
&:=& (\frac{\p \Phi}{\p y},\frac{\p \Phi}{\p k})
\\ \nonumber
&=& \int_{X(y,k)} f \frac{(-s_{yy},1)}{|D_x s_y|} d{\mathcal H}^{m-1}
+ \int_{X_\le(y,k)} (\frac{\p f}{\p y},0) d{\mathcal H}^m,
\end{eqnarray}
where $\mp s_{yy}^{(1\pm1)/2} /|D_x s_y|$
are the outward normal velocities of $X_\le(y,k)$ as either $y$ or $k$ is varied,
the other being held fixed.  If, in addition
$f\in C(Y, (C\cap W^{1,1})(X))$ then $\Phi \in C^1(U\setminus N)$.
\end{lemma}

\begin{proof}
Use the approximate Heavyside step function
$$
\phi_\epsilon(t) := \left\{
\begin{array}{lcr}
1 & {\rm if} & t \le 0, \cr
1-t/\epsilon & {\rm if} & t \in [0,\epsilon], \cr
0 & {\rm if} & t \ge \epsilon, 
\end{array}
\right.
$$
to define continuous functions
$$
h_\epsilon(y,k) = \int_{\partial^* X} \phi_\epsilon(s_y(x,y)-k) d{\mathcal H}^{m-1}(x)
$$
which are monotone in both $\epsilon>0$ and $k$,  and approximate
$$
h^+(y,k) = {\mathcal H}^{m-1}\left[\oX_\le(y,k) \cap \p^* X\right]
$$
pointwise from above.  This shows upper semicontinuity of $h^+$.
By symmetry,
$$
h^-(y,k) =  {\mathcal H}^{m-1}\left[\oX_<(y,k) \cap \p^* X\right]
$$
is lower semicontinuous.  Since
$$
(h^+-h^-)(y,k) =  {\mathcal H}^{m-1}\left[\oX(y,k) \cap \p^* X\right]
$$
we see $h^+ \ge h^-$, with equality holding outside of the closed set
$N$.  Since $X$ has finite perimeter and $\oX_\le(y,k)$ depends monotonically on $k$, 
for each $y$ there can only be countably many $k$ values with $(y,k) \in N$.

To compute the instantaneous rate of displacement of
$x \in X(y,k_0)$ as a function of $y$ near $y_0$, with $ k_0$ held fixed, 
 let $x(y)=\exp_{x_0} \lambda(y) \hat n_0$ denote the intersection of 
$X(y,k_0)$ with the geodesic parallel to outer unit normal $\hat n_0 := D_x s_y(x_0,y_0)/|D_x s_y|$ at $x_0 \in X(y_0,k_0)$,  meaning  $x(y)={x_0} + \lambda(y) \hat n_0$ in the Euclidean case.
Applying the implicit function theorem to $s_y(x(y),y) = k_0$,  we see a differentiable solution $\lambda$ exists near $y_0$ held fixed with 
$$
\lambda'(y_0) = \frac{- s_{yy}}{|D_x s_y|} (x_0,y_0).
$$
This gives the outer normal velocity $x'(y_0)=\lambda'(y_0) \hat n_0$ of $\partial X_\le(y,k_0)$ at $x_0$.
The corresponding normal velocity as $k$ is varied with $y_0$ held fixed is similar and even simpler to compute.

Now suppose $f \in C(Y,C\cap W^{1,1}(X))$ temporarily, with $f_y \in L^\infty$, and define
$$
\Phi_\epsilon(y,k) := \int_X f(x,y) \phi_\epsilon(s_y(x,y)-k) d{\mathcal H}^m (x).
$$
We claim the $\Phi_\epsilon$ are equi-Lipschitz approximations to $\Phi$.
The dominated convergence theorem and co-area formula \cite{EvansGariepy92} \cite{Federer69}
yield
\begin{equation}\label{approximate coarea}
D\Phi_\epsilon(y,k) 
= \int_X (f_y,0) \phi_\epsilon(s_y-k) d{\mathcal H}^m +
\frac{1}{\epsilon} \int_0^\epsilon dt \int_{X(y,k+t)} f \frac{(-s_{yy},1)}{|D_x s_y|} d{\mathcal H}^{m-1}.
\end{equation}
Letting $\nabla_X$ denote the divergence operator on $X$, while
 $\hat n_X$ and $\hat n_{X_=} := D_x s_y/|D_x s_y|$ denote the outward unit normals to $X$ and 
$X_\le(y,s_y(x,y))$ respectively,  the generalized Gauss-Green formula 
\cite[Proposition 5.8]{HofmannMitreaTaylor10} asserts
\begin{equation}\label{Gauss-Green}
\int_{X(y,k)} V \cdot \hat n_{X_=} d{\mathcal H}^{m-1} 
= \int_{X_\le(y,k)} \nabla_X \cdot V d{\mathcal H}^m 
-  \int_{\partial^* X \cap \overline{X_\le(y,k)}} V \cdot \hat n_X d{\mathcal H}^{m-1}
\end{equation}
for any continuous Sobolev vector field $V$ on $\overline X$.  We are interested in applying
this to vector fields which depend on an additional parameter $y \in Y$, whose components
in local coordinates lie in $C(Y,(C\cap W^{1,1})(X))$.
We claim this integral then depends continuously
on $(y,k) \in U\setminus N$,  and its magnitude is bounded throughout $U$ by
\begin{equation}\label{equi Lipschitz}
\left|\int_{X(y,k)} V \cdot \hat n_{X_=} d{\mathcal H}^{m-1} \right|
\le \|V\|_{W^{1,1}(X)} + \|V\|_{L^\infty(X)} {\mathcal H}^{m-1}(\partial^* X).
\end{equation}
The continuous dependence 
on $(y,k)$ follows from the dominated convergence theorem,
and the fact that 
$1_{X_\le(y',k')}$ converges to $1_{X_\le(y,k)}$  
Lebesgue almost everywhere on $X$ and ${\mathcal H}^{m-1}$-a.e.\ on $\partial^* X$ as 
$(y',k') \to (y,k) \not\in N$, in view of \eqref{no surface plateaus}.  In particular, we use this 
argument to show the last summand in
\begin{eqnarray*}
&& \left|\int_{X_\le(y',k')} \nabla_X \cdot V(x,y') d{\mathcal H}^m(x) -  \int_{X_\le(y,k)} \nabla_X \cdot V(x,y) d{\mathcal H}^m(x) \right|
\\ &\le& \int_{X_\le(y',k')} |\nabla_X \cdot (V(y')- V(y)) | d{\mathcal H}^m +  
\int_{X_\le(y,k) \Delta X_{\le(y',k')}} |\nabla_X \cdot V(y)| d{\mathcal H}^m
\end{eqnarray*}
vanishes in the limit $(y',k') \to (y,k)$,  where $\Delta$ denotes the symmetric difference of the domains of integration;
the other summand vanishes by the continuous dependence in $L^1(X)$ of $\nabla_X \cdot V$ on $y'$.

Choosing $V= \hat n_{X_=}$ in \eqref{Gauss-Green} demonstates the continuity of ${\mathcal H}^{m-1}[X(y,k)]$ on $U \setminus N$.
Alternately, choosing $V=fs_{yy}^{(1\pm1)/2} D_xs_y/|D_x s_y|^2$ for fixed $y\in\overline Y$,
we have just shown the inner integrals in \eqref{approximate coarea} depend
continuously on small $t$ as long as $(y,k) \not\in N$,  
and are locally uniformly bounded \eqref{equi Lipschitz}
with a constant depending only on ${\mathcal H}^{m-1}(\partial^* X)$ and 
\begin{equation}\label{old DPhi constant}
\sup_{y \in Y}\left\| \frac{f s_{yy}^{(1\pm1)/2} D_x s_y}{|D_x s_y|^2} \right\|_{(C \cap W^{1,1})\left({X}\right)} <\infty.
\end{equation}
Thus $\Phi_\epsilon$ converges locally uniformly on $U$ to a Lipschitz limit $\Phi_0$, and
$D\Phi_\epsilon$ converges pointwise on $U\setminus N$ to
\begin{equation}\label{DPhi_0}
D\Phi_0(y,k) 
= \int_{X_\le(y,k)} (\frac{\p f}{\p y},0) d{\mathcal H}^m +  \int_{X(y,k)} f \frac{(-s_{yy},1)}{|D_x s_y|} d{\mathcal H}^{m-1}.
\end{equation}
Since this surface intregral is controlled by the same constants, the preceding arguments show
$\Phi_0 \in C^{0,1}(U) \cap C^1(U \setminus N)$.  Since $N$ is Lebesgue negligible,
$\|\Phi_0 \|_{C^{0,1}(U)} = \|\Phi_0\|_{C^1(U \setminus N)}$; \eqref{DPhi_0} shows
the latter to be controlled by the listed constants.

On the other hand, the co-area formula \cite{EvansGariepy92} yields
\begin{eqnarray*}
|\Phi_\epsilon(y,k) - \Phi(y,k)| 
&\le& \int_{X_\le(y,k+\epsilon) \setminus X_\le(y,k)} |f(x,y)| d{\mathcal H}^m (x)
\\ &=& \int_0^\epsilon dt \int_{X(y,k+t)} \frac{|f|}{|D_x s_y|} d{\mathcal H}^{m-1}(x)  
\\ &\le& \epsilon \left\| \frac{f}{|D_x s_y|}\right\|_\infty 
\max_{(y,k) \in U} {\mathcal H}^{m-1}[X(y,k)]
\end{eqnarray*}
showing $\Phi=\Phi_0$ on $U$ and establishing the lemma for $f \in C^{0,1}(X \times Y)$.

We handle the case $f, f_y \in L^\infty(X \times Y)$ by approximation: mollification 
yields a sequence $f^\delta \in C^{0,1}(X \times Y)$ with $(f^\delta,f^\delta_y)$ uniformly
bounded and converging to $(f,f_y)$ pointwise a.e.\ as $\delta \to 0$. The dominated convergence
theorem asserts pointwise convergence of
$$
\Phi^\delta(y,k) := \int_{X_\le(y,k)} f^\delta(x,y) d{\mathcal H}^m(x)
$$
to $\Phi$.  On the other hand,  the version of the lemma
established above shows the $\Phi^\delta$ to be Lipschitz on $U$ with a constant 
independent of $\delta>0$.  Thus they converge uniformly to a limit $\Phi$ sharing 
the same Lipschitz constant.  The lemma also establishes \eqref{DPhi} on $U\setminus N$,
with $(\Phi^\delta,f^\delta)$ in place of $(\Phi,f)$.   We would like to use 
the dominated convergence theorem to pass to the limit $\delta \to 0$  in \eqref{DPhi}
for a.e.\ $(y,k)$. This works immediately when $f$ is continuous.  Otherwise,
let $Z \subset X \times Y$ denote the Lebesgue negligible set where 
$f^\delta$ fails to converge to $f$.  Fubini's theorem shows $Z(y) := \{x \in X \mid (x,y) \in Z\}$
has zero measure for a.e. $y \in Y$.  Applied to its indicator function,  the co-area
formula 
$$
0 = \int_X 1_{Z(y)} d{\mathcal H}^m = \int_\R dk \int_{X(y,k)} \frac{1_{Z(y)}}{|D_x s_y|} {d \mathcal H}^{m-1}
$$
then yields ${\mathcal H}^{m-1}$ negligibility of $Z(y) \cap X(y,k)$ for a.e.\ $k$.  For such $y$ and $k$,
the dominated convergence theorem permits passage to the $\delta \to 0$ limit in \eqref{DPhi}, 
to complete the proof.
\end{proof}

Our next goal is to establish Theorem \ref{T:non-nested}, 
which describes how the set of wives 
hypothetically paired with husband $y\in Y \subset \R$ move in response to changes in his type.



\begin{theorem}[Dependence of iso-husbands on husband type] 
\label{T:non-nested}
(a) Let $X \subset \R^m$ and $Y\subset \R$ be connected open sets of
finite perimeter, 
equipped with Borel probability measures $d\mu(x) =f(x) dx$ and $d\nu(y)=g(y)dy$ 
whose Lebesgue densities satisfy 
$\log f \in L^\infty(X)$ and $\log g \in L^{\infty}_{loc}(Y)$.
Assume $s\in C^2$ is non-degenerate
throughout $\overline {X} \times Y$, with all components of $(D_x s_y, s_{yy})$ in $C(Y; C \cap W^{1,1}(X))$  
Then the functions $k^\pm$ of Theorem~\ref{T:nested}(a) coincide. 
Moreover, $k:=k^\pm \in C^{0,1}_{loc}(Y)$ and 
$|k'(y)| 
 \le \|s_{yy}\|_{L^\infty(X)} + g(y) \| D_x s_y /f \|_{L^\infty(X)}/ {\mathcal H}^{m-1}(X(y,k(y)))$ a.e.

(b) If, in addition $\log f \in (C\cap W^{1,1})(X)$ 
and $\log g \in C^{0}_{loc}(Y)$ then $k$ is continuously differentiable
 outside the relatively closed set 
\begin{eqnarray}\label{bad set Z}
Z &:=& \{ y \in {\Y} \mid (y,k(y)) \in N {\rm\ from}\ \eqref{no surface plateaus} \}
\end{eqnarray}
and $k'(y)=-\frac{h_y}{h_k}(y,k(y))$ is given by \eqref{grad h} on $Y\setminus Z$.
As $y \in Y \setminus Z$ increases the outward normal velocity of
$X_\le(y,k(y))$ at $x \in X(y,k(y))$ is given by
$(k' - s_{yy})/|D_x s_y|$. 
If $\log g \in C(Y)$, then $k'(y)$ diverges to $+\infty$ at the endpoints of $Y$ unless ${\mathcal H}^{m-1}[X(y,k(y))]$ 
remains bounded away from zero in this limit.
\end{theorem}

\begin{proof}
Since the non-degeneracy of $s$ extends to $\oX \times Y$,  Proposition~\ref{P:indifference structure}
shows $\oX(y,k) := \oX_\le(y,k) \setminus \oX_<(y,k)$ to be the intersection with $\oX$ of an $m-1$ dimensional $C^1$ submanifold orthogonal to $D_x s_y(x,y) \ne 0$.  As in Theorem \ref{T:nested},  $k^\pm(y)$ represent
the maximal and minimal roots of the continuous function
\begin{equation}\label{h again}
h(y,k):=\mu[X_\le(y,k)]-\nu[(-\infty,y)], 
\end{equation}
which depends monotonically on $k$.
The open set $X_<(y,k^+(y))\setminus X_\le(y,k^-(y))$
carries none of the mass of $\mu$,  hence must be empty since $\log f$ is real-valued. 
Thus $k^+=k^-$ is continuous on $Y$ and $Z\subset Y$ is relatively closed.

Under the asserted hypotheses 
we claim $h\in C^{0,1}_{loc}(Y\times \R)$ and continuously differentiable outside of the set $N$ of zero measure from \eqref{no surface plateaus}.
Indeed, $h$ is locally Lipschitz on $Y \times \R$ 
according to Lemma~\ref{L:level set motion}, and its partial derivatives 
are given a.e. 
\begin{eqnarray}\label{grad h} 
h_k = \frac{\partial \Phi}{\partial k} =&  \displaystyle 
\int_{X(y,k)} f(x)   \frac{d{\mathcal H}^{m-1}(x)}{|D_x s_y(x,y)|}  \ge 0
& {\rm and}
\\ h_y = -g + \frac{\partial \Phi}{\partial y} = & \displaystyle- g(y) -
 \int_{X(y,k)} \frac{f(x) s_{yy}(x,y)}{|D_x s_y(x,y)|} d{\mathcal H}^{m-1}(x),
\nonumber \label{h_y}
\end{eqnarray}
adopting the notation $\Phi$ from \eqref{DPhi}.

In case (b) these derivative are continous outside of the closed negligible set $N$.
Since $h(y,k(y))=0$, if $h_k \ne 0$ the implicit function theorem then yields $k \in C^1_{loc}(Y \setminus Z)$,  
with $k' = - h_y/h_k$, and the stated bounds follow.  
In case (a),  the Clarke implicit function theorem yields the required bound on
$k\in C^{0,1}_{loc}(Y)$ \cite{Clarke83},
provided we can provide a positive lower bound for $h_k(y,k)$ a.e.\ in a neighbourhood of the graph 
$y=k(y)$ over compact subsets of $Y$.  Since $N\subset \oY \times \R$ has measure zero,
such a bound follows from \eqref{grad h}  provided we obtain a positive lower bound for 
${\mathcal H}^{m-1}[X(y,k)]$.  But this comes from lower semicontinuity of the relative
perimeter of $X_\le(y,k)$ in $X$ \cite{EvansGariepy92},  given that 
 $1_{X_\le(y',k')} \to 1_{X_\le(y,k)}$ Lebesgue a.e.\ as $(y',k') \to (y,k)$,  and the fact that
the $m-1$ dimensional submanifold
$X(y,k(y))$ is non-empty due to the connectedness of $X$.
 If ${\mathcal H}^{m-1}[X(y,k(y)]$ tends to zero at either endpoint of $Y$,  then
$h_k$ tends to zero but $h_y(y,k(y)) \to -g(y)$ does not,  showing $k'=-h_y/h_k$ diverges.

On compact subsets of $Y$,  
the outer normal velocity of $\partial X_\le(y_0,k(y_0))$ at $x_0$ comes from applying
Lemma \ref{L:level set motion} with $s(x,y)-\int^y k$ in place of $s$;
(the $\|s_{yy}\|_{W^{1,1}}$ bound hypothesized in that lemma is not needed for this particular claim).
\end{proof}


\begin{corollary}[Dynamic criterion for nestedness]
\label{C:dynamic nesting}
Adopting the hypotheses and notation of Theorem \ref{T:non-nested}(b):
If the model is nested then $k' - s_{yy} \ge 0$ for all $y\in Y\setminus Z$ and $x \in X(y,k(y))$,
with strict inequality holding at some $x$ for each $y$.  Conversely, if $Z=\emptyset$ and strict inequality holds
for all $y\in Y$ and $x \in X(y,k(y))$, then the model is nested.
\end{corollary}

\begin{proof}
Away from $Z$, continuous differentiability of $k=k^\pm$ and
the fact that the outward normal velocity of
$X_\le(y,k(y))$ at $x \in X(y,k(y))$ is given by $\frac{k'-s_{yy}}{|D_xs_y|}$ 
are established in Theorem \ref{T:non-nested}. 
Differentiating $\nu[(-\infty,y)] = \mu[X_\le(y,k(y))]$ gives
\begin{eqnarray}
g(y) &=& \frac{d}{dy} \int_{X_\le(y,k(y))} f(x) d^{m} x \cr
&=& \int_{X(y,k(y))} \frac{k'(y) - s_{yy}(x,y)} {|D_xs_y|} f(x) d{\mathcal H}^{m-1}(x),
\label{MA type}
\end{eqnarray}
at each $y\in Y\setminus Z$,  according to Lemma \ref{L:level set motion}.
If the model is nested, so that $y\in Y \mapsto X_\le(y,k(y))$ is increasing,
this velocity must be non-negative at each boundary point.  Positivity of $g$
in the formula above shows $k' - s_{yy}$ must  be positive at some boundary point.

Conversely,  if this velocity
is positive at each boundary point,  it means $X_\le(y,k(y))$ expands outwardly with $y$ 
at each boundary point $x$,  hence increases strictly with $y$ over the interval $Y$.
This confirms the model is nested.
\end{proof}

\begin{remark}[Boundary transversality of indifference sets]
\label{R:transversality}
If $y \in Z$ in \eqref{bad set Z}, the indifference set 
of $y$ must intersect $\partial X$
in a set of positive area. This can only happen if the normal $D_x s_y$ to this
indifference set coincides with the normal
to $\partial^* X$ almost everywhere on their intersection.

A sufficient condition for the set $Z \bigcap Y$ to be empty in \eqref{bad set Z} is therefore that 
the closure of any
potential indifference set  in $X$ of a husband type $y \in Y$ 
intersects $\partial X$ transversally.
For a Lipschitz domain $X$ this transversality implies the intersection is a Lipschitz
submanifold
of codimension 1 in $\partial X$ via the Clarke implicit function theorem \cite{Clarke83}.
\end{remark}

\begin{remark}[Monge-Amp\`ere type integrodifferential equation]
Identifying $k=v'$ and $X(y,k(y))= \f^{-1}(y)$, 
equation \eqref{MA type} should be compared to the Monge-Amp\`ere type equation
\begin{equation}\label{cMA}
g(y) =\pm  f(\f^{-1}(y))   \det [D^2 v - D^2_{yy} s] / \det \left[{D^2_{xy} s} \right]_{x=\f^{-1}(y)}
\end{equation}
which arises in the $n=m$ framework of Ma, Trudinger and Wang \cite{mtw}.
This comparison highlights the fact that positivity of the normal velocity $[v''-s_{yy}]_{y=\f(x)}$ 
plays a role in the multi-to-one dimensional problem analogous to strict ellipticity for \eqref{cMA}.
Indeed, Corollary \ref{C:non-zero speed} of the final section exploits this positivity 
to establish a Lipschitz bound on $\f$ and allow us to initiate our bootstrap to higher regularity.
In the same section, \eqref{distributional map gradient}  shows the Jacobian version of 
the balance condition \eqref{MA type} analogous to \eqref{1D Monge-Ampere} takes the 
form expected from the co-area formula, namely
$$
g(y) = \int_{F^{-1}(y)} \frac{f(x)}{|D\f(x)|} d{\mathcal H}^{m-1}(x).
$$
\end{remark}

\begin{corollary}[Unique splitting criterion for nestedness]
\label{C:unique splitting}
A model $(s,\mu,\nu)$ satisfying the hypotheses of Theorem \ref{T:non-nested}(b)
with $Z \bigcap Y=\emptyset$
is nested if and only each $x \in X$ corresponds to a unique $y \in Y$ splitting the population
proportionately, i.e. which satisfies
\begin{equation}\label{population split}
\int_{X_\le (y, s_y(x,y))} d\mu = \int_{-\infty}^y d\nu.
\end{equation}
In this case, the optimal map from $\mu$ to $\nu$ is given by $\f(x)=y$.
\end{corollary}

\begin{proof}
First assume the model is nested,  and suppose both $(x,y)$ and $(x',y')$ satisfy \eqref{population split},  
with $y\ne y'$,  so $k(y)=s_y(x,y)$ and $k(y') =s_y(x',y')$.
Taking $y<y'$ without loss of generality,  the hypothesis  $g>0$ of 
Theorem~\ref{T:non-nested} combines with the conclusion
$x \in X_\le(y,k(y)) \subset X_<(y',k(y'))$ of Theorem~\ref{T:nested} to imply 
$s_y(x,y')<k(y') = s_y(x',y')$.  Thus $y \ne y'$ implies $x \ne x'$ as desired.

Conversely,  suppose each $x \in X$ corresponds to a single $y$ splitting the 
population proportionately.  To prove the model is nested,  let us first check that
the outward velocity
$(k' - s_{yy})/|D_x s_y| \ge 0$ of $X_\le(y,k(y))$ at $x \in \partial X_\le(y,k(y))$ given
by Theorem \ref{T:non-nested} is non-negative for all $y\in Y$;
the moving boundary is a $C^{0,1}$ submanifold of $X$ according to the non-degeneracy of 
Proposition \ref{P:indifference structure}.

To derive a contradiction,  suppose $k'-s_{yy}<0$ for some $y_0 \in Y$ and $x_0 \in X(y_0,k(y_0))$.
Continuity shows this remains true for all nearby $y$ and $x \in X(y,k(y))$,  thus the boundary
of $X_\le(y,k(y))$ moves continuously inward near $x_0$ as $y$ increases through $y_0$.
Each point $x'$ in a sufficiently small neighbourhood $N_r :=B_r(x_0)  \setminus X_\le(y_0,k(y_0))$ 
therefore belongs to $X(y',k(y'))$ for some $y'<y_0$.
On the other hand,  
$$
\lim_{y\uparrow \bar y} \mu[X_\le(y,k(y))] = \lim_{y \uparrow \bar y} \nu[(\infty,y)] =1
$$
where $\bar y :=\sup Y$.
Since $\mu(N_r)>0$ by hypothesis, we conclude $N_r$ intersects $X_\le(y,k(y))$ for 
$y>y_0$ sufficiently close to $\bar y$.  Fix $x'$ from this intersection and let 
$y''$ denote the infimum of points $y>y_0$
such that $x' \in X_\le(y,k(y))$. Then $x' \in X(y',k(y'))\cap X(y'',k(y''))$ for $y''\in (y_0,\bar y)$,
with both $y'<y_0$ and $y''$ splitting the population proportionately at $x'$,  the desired contradiction.

This establishes $y\in Y \mapsto X_\le(y,k(y))$ is monotone non-decreasing.
It remains to confirm this monotonicity is strict.
Unless $X_\le(y,k(y)) \subset X_<(y',k(y'))$ for each pair of points $y<y'$ in $Y$,  some
$x$ lies in the boundary of both sets.  But then both $y$ and $y'$ split the population
proportionately at $x$.  This contradiction implies the model is nested,
and Theorem \ref{T:nested} implies the stable matching is given by $F(x)=y$.
\end{proof}

\begin{remark}[Methodological limitations are sharp]
This corollary confirms that nestedness is essentially a sharp condition for our method of solving
the problem to work;  in its absence the matching function $\f$ given by the procedure above fails to be 
well-defined: when the level sets $X(y,k)$ and 
$X(\bar{y},\bar{k})$ selected to match with $y\neq \bar{y}$ intersect at some $x$,
our construction would attempt to simultaneously assign both 
$\f(x):=y$ and $\f(x):=\bar{y}$.  
\end{remark}

\section{Examples}
\label{S:examples}

In the next sections, we address the higher regularity of optimal maps and potentials in nested problems.
Before doing so, we discuss several related examples demonstrating the phenomena
we subsequently analyze.   The first is nested, the second is not and the third explores the transition.  
All involve maximizing the bilinear surplus $s(x,y) = x \cdot y$ between two probability measures 
on $\R^m$,  which is equivalent to minimizing the quadratic cost $c(x,y) = \frac12|x-y|^2$.
In the first example,   the target measure will be supported on a segment; 
in the second and third, it will be supported on a circular arc.    These examples admit unique solutions
given explicitly by optimal maps, but demonstrate that such maps and the corresponding
potentials $u$ and $v$ will not necessarily
be smooth at the boundary. Although these examples are solved using classical methods and symmetry,  
they guide our intuition for what to expect from problems which do not admit explicit solution.

Besides arising frequently in different applications \cite{CullenPurser84} \cite{McCann97} \cite{RochetChone98}
\cite{FrischMatarreseMohayaeeSobolevskii02} \cite{HakerTannenbaum03},
the bilinear surplus / quadratic cost is the easiest objective to analyze when $mn > 1$. For these reasons
it has played a seminal role in theoretical developments  \cite{KnottSmith84} \cite{bren} \cite{Caffarelli92} \cite{Caffarelli96b}.
With this surplus function, target measures supported on lower dimensional sets were considered implicitly in 
\cite{CuestaMatran89} \cite{RuschendorfRachev90} \cite{McCann95} and explicitly in 
\cite{GM2}.  The basic result dating to these early works is that the payoff functions $u,v$ which
optimize \eqref{dual} may be taken to be convex on $\R^m$,  and that $\gamma$ optimizes \eqref{MK}
if and only if it is supported in the subdifferential 
$$
\partial u := \{ (x,y) \in \R^{m} \times \R^m \mid u(z) \ge u(x) + y \cdot (z-x) \quad \forall z \in \R^m\}
$$
of some convex $u:\R^m \longrightarrow \R \cup \{+\infty\}$.  In particular,  the optimal $\gamma$
is unique provided $\mu$ vanishes on Lipschitz hypersurfaces \cite{McCann95}, 
or at least on those hypersurfaces parameterized by convex differences \cite{GM} \cite{Gigli11}.

\begin{example}[From convex volumes to segments]
\label{E:segment}
Fix  $s(x,y) = x \cdot y$ on $\R^m \times \R^m$, and
consider transporting volume from a 
convex body $X \subset \R^m$ 
to a uniform meassure on a subsegment 
of the $x_1$-axis. 
In this case the convex payoff function $u(x_1,\ldots,x_m) = U(x_1)$ turns out to depend on $x_1$ only,
and the iso-husband sets consist of (hyper-)planes of constant $x_1$;
convexity ensures that --- apart from the two supporting hyperplanes --- these hit $\partial^* X$ transversally,
so $Z$ is empty.
Since these hyperplanes do not intersect each other, Corollary \ref{C:unique splitting} shows
this problem is nested.  The optimal map $y=\f(x)=(U'(x_1),0,\dots,0)$ depends monotonically on $x_1$,  
and can be found by solving a problem in single variable calculus, analogous to \eqref{non-decreasing map}.
Taking $X$ to be a ball leads to elliptic integrals not explicitly soluble,  
but transporting the solid paraboloid 
$$X := \{x \in \R^m \mid \frac12 \sum_{i=2}^{m} x_i^2 < x_1 < const\}$$ 
to a segment $Y = (0, L\hat e_1)$ of the appropriate length 
yields the optimal map  
$\f(x_1,\ldots,x_m) = (x_1^{(m+1)/2},0,\ldots,0)$ and potentials 
$u(x) = \frac{2}{m+3} x_1^{1 + (m+1)/2}$ and $v(y) =  \frac{m+1}{m+3} y^{1 + 2/(m+1)}$ explicitly.
Their behaviour near the origin shows we cannot generally expect $v$ to be better than $C^{1,\frac{2}{m+1}}$; 
similarly, we expect $F \in C^{\frac m 2, \frac 1 2}$ 
and $u \in C^{\frac{m} 2 + 1, \frac 1 2}$ 
to be sharp H\"older exponents near the origin, at least for $m$ even and sufficiently convex domains $X$.   
This lack of $C^2$ smoothness of $v$ at the boundary of $Y$ was predicted by Theorem \ref{T:non-nested};
like the lack of  higher order smoothness of $F$ and $u$ 
at the boundary of $X$, it directly reflect the unequal dimensions of the source and target.

In fact, in the absence of strong convexity of $X$, we cannot even expect this much smoothness up to the boundary.  If we consider instead
$$X := \{x \in \R^m \mid \frac12 \Big(\sum_{i=2}^{m} x_i^2 \Big)^k< x_1 < const\}$$ 
for $k\geq 1$, the optimal map takes the form $F(x) = Cx_1^{1+\frac{m-1}{2k} }$.  By choosing $k$ large enough, this shows that we cannot generally expect  $F \in C^{1, \alpha}$ up to the boundary for any $\alpha >0$, unless we assume $X$ has some uniform convexity.
\end{example}
Note that the surplus in the preceding example is of pseudo-index form, with index function $I(x_1,...,x_m) =x_1$.  Generally speaking, the bilinear surplus $s(x,y) =x\cdot y$ on an open set $X \subseteq \R^m$ and a smooth curve $Y \subset \R^m$ is of index form if and only if $Y$ is contained in a line.  The examples below treat the case where $Y$ is a circular arc and therefore $s$ is not of pseudo-index  form.  Special cases of these examples were studied in \cite{P2}, before the notion of nestedness had  been formulated.
\begin{example}[From punctured ball to punctured circle]
\label{E:circle}
Let 
$s(x,y) = x \cdot y$ on $\R^m \times \R^m$ %
and consider transporting 
volume $\mu= \frac{1}{{\mathcal H}^m[X]} {\mathcal H}^m$ from the 
punctured
 ball $X:=\{ x \in \R^m \mid 0< |x|<1\}$ to 
arclength $\nu = \frac1{{\mathcal H}^1[Y]} {{\mathcal H}^1}$
on the punctured circle $Y = \{ y \in \R^m \mid y_1^2 + y_2^2 =1, \quad
y_3 =\cdots = y_m =0<y_1+1\}$.  Since the map $\f(x) = \frac{(x_1,x_2)}{(x_1^2+x_2^2)^{1/2}}$ 
pushes $\mu$ forward to $\nu$,
and its graph lies in the subdifferential of the convex function $u(x)=\sqrt{x_1^2 + x_2^2}$, 
 we see $\f$ is optimal. Moreover $u$ and its Legendre transform
$$
v(y) = \left\{
\begin{array}{cc}
0 &{\rm if}\ |x|\le 1\ {\rm and}\ x_3 = \cdots = x_m=0,\\
+\infty &{\rm else},
\end{array}
\right.
$$
minimize the dual problem \eqref{dual}.  Although $v$ is smooth on $Y$, 
note $F$ and $Du$ fail to be smooth on the codimension $2$
submanifold $x_1=0=x_2$, where they are discontinuous.  The iso-husband sets $\f^{-1}(y)$ 
consist of connected 
hypersurfaces in $X$ bounded by this singular set;  they intersect the boundary of $X$ transversally.
On the other hand, parameterizing the punctured circle
$Y$ using $-\pi<\theta<\pi$,  the surplus function takes the form 
$s(x,\theta) = x_1 \cos \theta + x_2 \sin \theta$,  and the payoff $v(\theta)=0$. The 
potential indifference sets $X(\theta,k) := \{x\in \overline X \mid x_1 \sin\theta - x_2\cos \theta =k\}$ 
are connected. The iso-husband set $F^{-1}(\theta)$ occupies precisely half of
$X(\theta,v'(\theta))$ 
which shows the model is not nested.  

This example illustrates a phenomenon which our subsequent analysis shows to be
intimately connected with nestedness: 
smoothness often holds for the payoff on the lower dimensional space $Y$ even when it fails for the
optimal map and payoff on the higher dimensional space $X$. 
\end{example}

For the ball $X$,  Examples \ref{E:segment} and \ref{E:circle} represent limits of a continuum of examples
consisting of circular arcs $Y$ of  angle $|\theta|<\theta_0 $ and radius $1/\theta_0$.  Having less symmetry, they are not explicitly solvable, but it is natural to expect they remain nested for $\theta_0$ less than some critical value 
$\theta_c \in (0,\pi)$,  and become un-nested otherwise.  For $\theta_0>\theta_c$ it is not at all 
obvious how the singularities in $\f$ may be located  (though they are characterized implicitly
through duality).  Note that the analysis of Example \ref{E:circle} extends equally
well to the case where $X$ is a ball or spherical shell $\{ 0< r<|x|<R\}$ instead of a punctured ball;
in the latter case $\f$ will be smooth.

One last example illustrates the transition from nestedness to non-nestedness.
In this example we also see the continuity of the map $\f:X \longrightarrow Y$ 
shown in Theorem \ref{T:nested} need not extend to the closed set $\oX$.

\begin{example}[From pie slice to circular arc]
Fix $0<\theta_0<\pi$ and $r_0>0$.
Let $\mu$ be uniform on the pie-shaped region 
$X \subset \R^2$ described by $r<r_0$ and $|\theta|< \theta_0$ in polar coordinates.
Let $\nu$ be uniform over the circular arc $r=1$ and $|\theta|<\theta_0$.
The optimal map $\f$ and potentials $u$ and $v$ are the same as in Example \eqref{E:circle},
by the same arguments (or by restriction).
However 
this model is nested if and only if $\theta_0\le \pi/2$;
the {\em if} implication is shown using Corollary \ref{C:unique splitting},
and the {\em only if} by the logic of Example \ref{E:circle}.
\end{example}

\section{Regularity of husband's payoff}

Regularity of the map $\f$ and payoff functions $u$ and $v$ is 
a notoriously delicate question which has recieved considerable attention in case $m=n>1$, 
reviewed in~\cite{DePhilippisFigalli14} and \cite{Mc14}.
Priori to the work of Ma, Trudinger and Wang~\cite{mtw}, 
Villani \cite{V} had described it 
as {\em ``Without any doubt,  the main open problem in the field''}.
When $m>n$, very little is known~\cite{P2}.  For $m>n=1$,  Theorem \ref{T:nested} 
shows nestedness implies continuity of $F$ on the domain interior $X=X^0$; Theorem \ref{T:non-nested} 
combines with Remark \ref{R:transversality} to give conditions guaranteeing $v \in C^2_{loc}(Y^0)$;  
Example \ref{E:segment} shows we  cannot expect $k=dv/dy$ to have a H\"older exponent larger
than $\frac{2}{m+1}$ at the endpoints of $Y$.

The present section is devoted to the following theorem,  which shows that higher regularity of
the husband's payoff $v$ 
can be inferred from that of $(s,\mu,\nu)$ and $\p X$ under the same boundedness and transversality
conditions required by Theorem \ref{T:non-nested} and Remark \ref{R:transversality}.
The strategy of the proof is to iteratively combine our hypotheses with geometric measure theoretic 
level-set techniques to deduce sufficient smoothness of the function $h(y,k)$ from \eqref{h again} 
whose level sets implicitly define $k=dv/dy$.  
The result then follows from the implicit function theorem as in Theorem \ref{T:non-nested}.
We exercise care to localize the hypotheses in $(x,y)$ where possible.
We discuss smoothness of the map $\f$ and wives' payoffs $u$ in a subsequent section.


\begin{theorem}[Regularity of the husband's payoff]
\label{T:h regular}
Fix an integer $r \ge 1$. 
Under the hypotheses of Theorems \ref{T:nested}(b)
and \ref{T:non-nested}(b), suppose
there is an interval $Y'=(y_0,y_1) \subset Y$  
such that ${X'} \cap \partial X \in C^1$ intersects $\overline {X(y,k(y))}$ transversally
for all $y \in \overline{Y'}$, where $X' = \bigcup_{y\in {Y'}} \overline{X(y,k(y))}$.
Then $\|k\|_{C^{r,1}(Y')}$ is controlled by the following quantities, all assumed 
positive and finite: $\| \log f \|_{C^{r-1,1}(X')}$, $\| \log g\|_{C^{r-1,1}(Y')}$, $\|s_y\|_{C^{r,1}(X' \times Y')}$, 
 $\|\hat n_X\|_{(C^{r-2,1}\cap W^{1,1})(X' \cap \p X)}$,
 ${\mathcal H}^{m-1}[\partial^* X]$,
\begin{eqnarray}
\label{modulus of level set area}
\inf_{y \in Y'} {\mathcal H}^{m-1}\left[X(y,k(y))\right] &\quad& \mbox{\rm (proximity to ends of $Y$)},
\\ \label{modulus of n-d}
\inf_{x \in X', y \in Y'} |D_x s_y(x,y)| && \mbox{\rm (non-degeneracy)},
\\ \label{modulus of transversality}
\inf_{x \in X' \cap \partial X, y\in Y'} 1- (\hat n_X \cdot \hat n_{X_=})^2
&& \mbox{\rm (transversality)}
\end{eqnarray}
where $\hat n_{X_=}(x,y)=D_x s_y/|D_x s_y|$,
and ${\mathcal H}^{m-2} \left[\overline{X(y_0,k(y_0))} \cap \p X \right]$.
\end{theorem}

To establish this theorem,  we study the motion of the interior $X_\le(y,k)$  and boundary 
$X_\le(y,k) \cap \p X$ level sets of $s_y$ with respect to changes in $y$ and in $k$.  The normal velocities
for this motion are given by
\begin{equation}\label{normal velocity}
V^\pm(x,y) := \mp \frac{(s_{yy})^{(1\pm 1)/2}}{|D_x s_y|} \hat n_{X_=}, \qquad  \hat n_{X_=} (x,y):= \frac{D_x s_y}{|D_x s_y|},
\end{equation}
and
\begin{equation}\label{boundary velocity}
V^\pm_{\partial X}(x,y) := \frac{V^\pm \cdot \hat n_{X_=}}{\sqrt{1-(\hat n_X \cdot \hat n_{X_=})^2}} \hat n_\partial, \quad \hat n_\partial(x,y)  :=\frac{\hat n_{X_=} - (\hat n_{X_=} \cdot \hat n_X) \hat n_X}{\sqrt{1 - (\hat n_X \cdot \hat n_{X_=})^2}},
\end{equation}
respectively.  We shall also employ the divergence operators $\nabla_X$ on $X$ and $\nabla_{\p X}$ on $\p X$.

We remark the outward normal velocity of $X_\le \cap \partial X$ is given by $V^\pm_{\p X}$,
since $\hat n_\partial$ is the outward unit normal to $X_\le \cap \partial X$ in $\partial X$, and
$V^\pm$ coincides with the projection of $V^\pm_{\partial X}$ onto $\hat n_{X_=}$. As it will play an important role in what follows, let us also
remark on the smoothness of the domain $X$ and its boundary divergence operator.  
Any $\partial X \in C^1$
 can be parameterized locally as a graph of a function $w \in C^1(\R^{m-1})$;
in this parameterization its metric tensor takes the form $g = I + Dw \otimes Dw$,
and $\nabla_{\partial X} \cdot W = \partial_i W^i + \Gamma^{i}_{i\ell} W^\ell$,  where the Christoffel symbol 
$\Gamma^i_{i\ell}$ involves first derivatives of $g$.  Thus we need $\hat n_X \in W^{1,1}(\p X)$ 
to define $\nabla_{\partial X}$,
and if $\partial X \in C^{r-1,1}$ then $\nabla_{\partial X} \cdot W \in C^{r-3,1}$ provided the vector field $W$ is smooth enough (say $W \in C^{r-2,1}$), with $C^{-1,1} = L^\infty$ conventionally.  

Our first lemma shows the size of the indifference sets ${X(y,k(y))}$ and their boundaries
are controlled uniformly on $(y_0,y_1)$ by the constants listed in the theorem.

\begin{lemma}[Size of indifference sets]\label{L:sizes}
Under the hypotheses of Theorem \ref{T:non-nested}(a), 
the area $A(y) = {\mathcal H}^{m-1}[X(y,k(y))]$ 
of the indifference set is controlled by
\begin{equation}\label{level set upper bound global}
A(y) 
\le
\frac{\|s_y\|_{W^{1,1}(X)}}{ \displaystyle \inf_{x  \in X} |D_x s_y(x,y)|} + {\mathcal H}^{m-1}[\p^* X].
\end{equation}
Under the hypotheses of Theorem \ref{T:h regular},  it  satisfies
\begin{equation}\label{level set upper bound local}
\sup_{y \in Y'} A(y)
\le
\frac{\|  D_xs_y \|_{C^{0,1}(X' \times Y')} }{\displaystyle \inf_{(x,y)  \in X' \times Y'} |D_x s_y|} {\mathcal H}^m(X') + {\mathcal H}^{m-1}(X' \cap \partial X)
+ \inf_{y \in Y'}A(y)
\end{equation}
while the area of its boundary $B(y):= {\mathcal H}^{m-2}\left[\overline{X(y,k(y))}\cap \p X\right]$ satisfies
\begin{equation}\label{modulus of boundary level set}
\sup_{y \in Y'} B(y) 
\le B(y_0) + \|\nabla_{\p X} \cdot \hat n_\p \|_{L^1(X' \cap \p X)},
\end{equation} 
hence both are controlled by the constants named in that theorem. 
\end{lemma}

\begin{proof}
Let  $X_\le^y :=\overline{X_\le(y,k(y))}$ and $X_=^y := \overline{X(y,k(y))}$. We claim bounds on
\begin{eqnarray}
A(y) &:=&\int_{X_=^y} ( \hat n_{X_=}(x,y) \cdot \hat n_{X_=}(x,y)) d{\mathcal H}^{m-1}(x) \qquad {\rm and}
\\ B(y) &:=&\int_{X_=^y\cap\partial X}  (\hat n_{\partial} (x,y)\cdot \hat n_{\partial}(x,y)) d{\mathcal H}^{m-2}(x).
\end{eqnarray}
Note both depend continuously on $y \in Y$ by Lemma \ref{L:level set motion} 
and Theorem \ref{T:non-nested}.

The generalized Gauss-Green theorem \eqref{Gauss-Green} yields
$$
A(y) = \int_{X_\le^y} (\nabla_X \cdot \hat n_{X_=} 
) d{\mathcal H}^m
 - \int_{X_\le^y \cap \p^* X} (\hat n_{X_=} 
\cdot \hat n_X) d{\mathcal H}^{m-1}
$$
from which \eqref{level set upper bound global} follows immediately.  Now suppose the infimum 
\eqref{modulus of level set area} is attained at $z \in \overline{Y'}$.  If $z \le y$ 
the Gauss-Green theorem yields
\begin{eqnarray*}
A(y) &=& \int_{X_\le^y \setminus X_\le^{z} } (\nabla_X \cdot \hat n_{X_=} 
) d{\mathcal H}^m
 - \int_{[X_\le^y \setminus X_\le^z] \cap \partial X} (\hat n_{X_=} 
\cdot \hat n_X) d{\mathcal H}^{m-1}
\\ && + \int_{X_=^z} \hat n_{X_=}(x,y) \cdot \hat n_{X_=}(x,z) d{\mathcal H}^{m-1}(x),
\end{eqnarray*}
where the nestedness $X^z_\le \subset X^y_\le$ assumed in Theorem \ref{T:h regular} has been used.
If $z>y$ a similar formula holds,  with the roles of $y$ and $z$ interchanged.
In either case \eqref{level set upper bound local} follows.

Similarly, the  Riemannian version \eqref{Gauss-Green}
of generalized Gauss-Green theorem asserts
$$
B(y) = \int_{[X_\le^y  \setminus X_\le^{y_0}]\cap\partial X} (\nabla_{\partial X} \cdot \hat n_{\p}) d{\mathcal H}^{m-1} 
+ \int_{X_=^{y_0}  \cap \p X} \hat n_\p(x,y) \cdot \hat n_\p(x,y_0) d{\mathcal H}^{m-2},
$$
which implies \eqref{modulus of boundary level set}.  It is easy to see the quantities appearing in
\eqref{level set upper bound local}--\eqref{modulus of boundary level set} are controlled by those
listed in Theorem \ref{T:h regular}.  
%
\end{proof}


\begin{remark}
If $\hat n_X \in C^{0,1}(\p X)$,  a similar argument using a Kirszbraun extension to $\p X$ 
can be used to bound $B(y)$ independently of $B(y_0)$.
\end{remark}

We next adapt Lemma \ref{L:level set motion} to the differentiation of boundary fluxes.

\begin{lemma}[Derivative of a flux through a moving boundary] 
\label{L:differentiating fluxes}    \ \\
Suppose $X' \subset X \subset \R^m$, $Y' \subset Y\subset \R$, $(s,f,g)$ and $k$ 
satisfy the hypotheses of Theorem \ref{T:h regular} for some $r \ge 1$.
Choose a neighbourhood $U\subset Y \times \R$ such that $X(y,k) \subset X'$ 
for all $(y,k) \in U$.
If $a:X' \times Y' \longrightarrow \R^m$ 
is Lipschitz, then 
\begin{equation}\label{flux}
\Phi(y,k) := 
\int_{X(y,k)} a(x,y) \cdot \hat n_{X_=}(x,y) d{\mathcal H}^{m-1} (x) 
\end{equation}
is Lipschitz on $U$, with partial derivatives $(\Phi^+,\Phi^-) := (\frac{\p \Phi}{\p y},\frac{\p \Phi}{\p k})$ given
a.e.\ by
\begin{eqnarray}\label{derived flux}
\Phi^\pm(y,k) &=& \int_{X(y,k)} [(\nabla_X \cdot a)V^\pm + \frac{1\pm1}{2}\frac{\p a}{\p y}] \cdot \hat n_{X_=} d{\mathcal H}^{m-1} 
\\ &&- \int_{\overline{X(x,y)} \cap \p X}  (a \cdot \hat n_X) (V^\pm_{\p X} \cdot \hat n_\p)d{\mathcal H}^{m-2}.
\nonumber
\end{eqnarray}
Similarly, if $r \ge 2$ and $b:(X' \cap \p X) \times Y' \longrightarrow T\p X$ 
denotes a jointly Lipschitz family of sections of the tangent bundle of $X'\cap \p X$, 
then 
\begin{equation}\label{boundary flux}
\Psi(y,k) := 
\int_{X(y,k)\cap \p X} b(x,y) \cdot \hat n_{\p}(x,y) d{\mathcal H}^{m-2} (x) 
\end{equation}
is Lipschitz on $U$, with partial derivatives $(\Psi^+,\Psi^-) := (\frac{\p \Psi}{\p y},\frac{\p \Psi}{\p k})$ given
a.e.\ by
\begin{eqnarray}\label{derived boundary flux}
\Psi^\pm(y,k) &=& \int_{\overline{X(y,k)} \cap \p X} [(\nabla_{\p X} \cdot b)V_{\p X}^\pm + \frac{1\pm1}{2}\frac{\p b}{\p y}] \cdot \hat n_{\p} \ d{\mathcal H}^{m-2}. 
\end{eqnarray}
\end{lemma}

\begin{proof}
To see that the proof of \eqref{derived boundary flux} is completely analogous to the proof of \eqref{derived flux},
let us argue the latter for $\R^m$ replaced by a complete 
Riemannian manifold $M$ with $g_{ij} \in C \cap W^{1,1}$, such as $M=\p X$.  
Note $\p X \in C^{1,1}$ if $r\ge 2$;
its lack of boundary accounts for the comparative simplicity of \eqref{derived boundary flux}
relative to \eqref{derived flux}.

Choose a $C^{1,1}$ smooth family 
$a_\epsilon:X' \times Y' \longrightarrow TX'$ of sections of the tangent bundle
converging to $a$ in Lipschitz norm,
so that both $\nabla_X \cdot a_\epsilon$ and $\frac{\p}{\p y} (\nabla_X \cdot a_\epsilon) \in L^\infty$.
Define $\Phi_\epsilon$ by \eqref{flux} with $a$ replaced by $a_\epsilon$.
Let $Y'=(y_0,y_1)$,  so that $X \cap \p X'= X(y_0,k(y_0)) \cup X(y_1,k(y_1))$.
The generalized Gauss-Green theorem \eqref{Gauss-Green} theorem yields
\begin{eqnarray*}
\Phi_\epsilon(y,k) &=& \int_{X_\le(y,k) \setminus X_\le(y_0,k(y_0))} 
\nabla_X \cdot a_\epsilon d{\mathcal H}^m  
\\&&-\int_{[X_\le(y,k)\setminus X_\le(y_0,k(y_0))] \cap \p X} 
a_\epsilon \cdot \hat n_{X} d{\mathcal H}^{m-1}
\\&&+ \int_{X(y_0,k(y_0))} a_\epsilon(x,y) \cdot \hat n_{X_=}(x,y_0) d{\mathcal H}^{m-1}(x)
\end{eqnarray*}
Lemma \ref{L:level set motion} then asserts $\Phi_\epsilon$ is Lipschitz on $U$, with
\begin{eqnarray*}
\Phi_\epsilon^+
&=& \int_{X(y,k)} (\nabla_X \cdot a_\epsilon) V^+ \cdot \hat n_{X_=} d{\mathcal H}^{m-1} 
+\int_{X_\le(y,k) \setminus X_\le(y_0,k(y_0))} \nabla_X \cdot \frac{\p a_\epsilon}{\p y} d{\mathcal H}^m
\\ && 
-  \int_{\overline{X(x,y)} \cap \p X}  (a_\epsilon \cdot \hat n_X) (V^+_{\p X} \cdot \hat n_\p)d{\mathcal H}^{m-2}
-  
\int_{[X_\le(y,k) \setminus X_\le(y_0,k(y_0))]\cap \p X} \hat n_X \cdot \frac{\p a_\epsilon}{\p y} d{\mathcal H}^m
\\ &&+ \int_{X(y_0,k(y_0))} \frac{\p a_\epsilon}{\p y}(x,y) \cdot \hat n_{X_=}(x,y_0) d{\mathcal H}^{m-1}(x)
\\ &=& \int_{X(y,k)} [(\nabla_X \cdot a_\epsilon) V^+ + 
\frac{\p a_\epsilon}{\p y} ] \cdot \hat n_{X_=} d{\mathcal H}^{m-1}
 -  \int_{\overline{X(x,y)} \cap \p X}  (a_\epsilon \cdot \hat n_X) (V^+_{\p X} \cdot \hat n_\p)d{\mathcal H}^{m-2}
\end{eqnarray*}
a.e.; the generalized Gauss-Green identity has again been used.  
A similar formula holds for $\Phi_\epsilon^-$ with $(V^+,V^+_{\p X}, \frac{\p}{\p y})$ replaced by
$(V^-,V^-_{\p X}, \frac{\p}{\p k})$.  In particular,
$\|\Phi_\epsilon\|_{C^{0,1}(Y \times \R)}$ can be bounded independently of
$\epsilon$ by $\|w\|_{C^{0,1}}$, 
$\|(V^\pm,V^\pm_{\p X})\|_{C^0}$, \eqref{level set upper bound} and \eqref{modulus of boundary level set}.
Since $\|a_\epsilon - a \|_{C^{0,1}(X \times Y; TX)} \to 0$ we can pass to the 
$\epsilon \to 0$ limit to bound $\|\Phi\|_{C^{0,1}(Y \times \R)}$ using the same quantities,
and obtain \eqref{derived flux} from the dominated convergence theorem.

When $r\ge 2$,  approximating $b$ by $b_\epsilon$ analogously, the same reasoning yields $\Psi_\epsilon$
Lipschitz on $U$ with
\begin{eqnarray*}
\Psi_\epsilon^\pm(y,k)
&=&
 \int_{\overline{X(y,k)}\cap \p X} [(\nabla_{\p X} \cdot b_\epsilon) V_{\p X}^\pm + \frac{1\pm1}{2} 
\frac{\p b_\epsilon}{\p y} ] \cdot \hat n_{\p} \ d{\mathcal H}^{m-2}
\end{eqnarray*}
a.e.  
Since $\|b_\epsilon - b \|_{C^{0,1}([X' \cap \p X] \times Y'; T\p X)} \to 0$,
we see $\|\Psi_\epsilon\|_{C^{0,1}(U)}$ is controlled as before by $\|z\|_{C^{0,1}}$, 
$\|V^\pm_{\p X}\|_{C^0}$ and \eqref{modulus of boundary level set},  permitting passage
to the $\epsilon \to 0$ limit.
\end{proof}

\begin{corollary}[Iterated differentiation of fluxes]\label{C:iterated derivatives}
Under the hypotheses and notation of Lemma \ref{L:differentiating fluxes},
if $r \ge 2$ or $b=0$ then
$$
(\Phi+\Psi)^\pm =
\int_{X(y,k)} a^\pm \cdot \hat n_{X_=} d{\mathcal H}^{m-1} + 
\int_{\overline{X(y,k)} \cap \p X} b^\pm \cdot \hat n_{\p} d{\mathcal H}^{m-2} 
$$
a.e.\ on $U$, where $\left({a^\pm \atop b^\pm}\right) =  A_\pm\left({a \atop b}\right)$
are given by the first-order differential operators
$$
 A_\pm
\left(
\begin{array}{c}
a \cr
b 
\end{array}
\right)
:= \left(
\begin{array}{c}
(\nabla_X \cdot a) V^\pm + \frac{1\pm 1}{2} \frac{\p a}{\p y} \phantom{-a}\cr
 (\nabla_{\p X} \cdot b -  \hat n_X \cdot a )  V^\pm_{\p X} + \frac{1\pm 1}{2}\frac{d b}{d y}  
\end{array}
\right).
$$
For each integer $0 \le i \le r-2$, the operator
$A_\pm:B_{i} \longrightarrow B_{i-1}$ gives a bounded linear transformation
from the Banach space 
$$
B_i:=C^{i,1}(X' \times Y'; \R^m) \oplus C^{i,1}([X' \cap \p X] \times Y';T\p X)
$$ 
of $C^{i,1}$ families of sections of the tangent bundle to $B_{i-1}$;
its norm 
$$
\|A_\pm\|_{B_i \to B_{i-1}} \le C_{m,i} (1+ \|V^\pm\|_{C^{i-1,1}(X' \times Y')} 
+ (1+ \|\hat n_X\|_{C^{i-1,1}(X' \cap \p X)} )
\|V^\pm_{\p X})\|_{C^{i-1,1}([X' \cap\p X] \times Y')}
)
$$
is controlled by $\|s_y\|_{C^{i,1} (X \times Y')}$, $\|\hat n_X\|_{C^{i-1,1}(\p X)}$, 
\eqref{modulus of n-d} and \eqref{modulus of transversality}.
Similarly,   
$$\|A_\pm\|_{C^{r-1,1}(X \times Y') \oplus \{0\} \to B_{r-2}} \le 
C_{m,r-1} (1+\|V^\pm\|_{C^{r-2,1}(X \times Y')} + \|V^\pm_{\p X} \otimes \hat n_X\|_{C^{r-2,1}(\p X \times Y')})
$$
is controlled by $\|s_y\|_{C^{r-1,1}(X \times Y')}$, $\|\hat n_X\|_{C^{r-2,1}(\p X)}$
and \eqref{modulus of n-d}.
\end{corollary}

\begin{proof}The corollary follows directly from Lemma \ref{L:differentiating fluxes}.
 The norms of the 
first order differential operators $A_\pm$ are elementary to estimate.
\end{proof}
\medskip

Finally, we are ready to iterate derivatives using the preceding lemma and its corollary
to deduce Theorem \ref{T:h regular} from Theorem \ref{T:non-nested}. 
\medskip

\begin{proof}[Proof of theorem]
Theorem \ref{T:non-nested} asserts $h(y,k) := \mu[X_\le(y,k)] - \nu[(-\infty,y)]$
is continuously differentiable on the set $U$ of Lemma \ref{L:differentiating fluxes}.  Apart from the additive term $g(y)$,
 its partial derivatives $(h_+,h_-):=(g + \frac{\p h}{\p y},\frac{\p h}{\p k})$ 
are given in \eqref{grad h} as flux integrals
$$
h_\pm(y,k) = %
\int_{X(y,k)} f V^\pm \cdot \hat n_{X_=} d{\mathcal H}^{m-1}.
$$
over the indifference set $X(y,k)$.
Taking $a=fV^\pm \in C^{r-1,1}(X' \times Y'; \R^m)$ and $b=0$, 
Corollary \ref{C:iterated derivatives} allows us to compute and bound
$r$ derivatives of $h_\pm$ on $U$ iteratively.

Taking $0 \le i\le j \le r$, in the
notation of that corollary,
\begin{equation}\label{r derivatives}
\frac{\p^{j} h_\pm}{\p k^{j-i} \p y^i} = \int_{X(y,k)} a^j_i \cdot \hat n_{X_=} d{\mathcal H}^{m-1} +
\int_{\overline{X(y,k)} \cap \p X} b^j_i \cdot \hat n_{\p}\  d{\mathcal H}^{m-2} 
\end{equation}
where $\left( {a^j_{i} \atop b^j_{i}} \right) := A_-^{j-i} A_+^i \left( {f V^\pm \atop 0} \right) \in B_{r-j-1}$.
Furthermore, $\| \left( {a^{j}_i \atop b^{j}_i} \right) \|_{C^{r-1-j,1}}$ is controlled by 
quantities named in the theorem: $\|f\|_{C^{r-1,1}(X')}$, $\|s_y\|_{C^{r,1}(X' \times Y')}$, 
$\|\hat n_X\|_{C^{r-2,1}(X' \cap \p X)}$, \eqref{modulus of level set area}--\eqref{modulus of transversality},
${\mathcal H}^{m-1}[\p^* X]$ 
and ${\mathcal H}^{m-2}[\overline{X(y_0,k(y_0)} \cap \p X]$, in view of Lemma \ref{L:sizes}.
Since all but the final derivatives provided by the Corollary are continuous we may order them as convenient,
and discover $\|h\|_{C^{r,1}(U)}$ is controlled by the listed quantities.    
 
Now $k$ solves $h(y,k(y))=0$, so  the implicit function theorem
result provided by Theorem \ref{T:non-nested} shows $k$ 
inherits the same smoothness as $h$.  Using \eqref{grad h} to bound 
$h_k(y,k(y)) > 0$ away from zero by the product of
$\|D_x s_y /  f\|_{L^\infty}$ with \eqref{modulus of level set area},
we deduce $\|k\|_{C^{r,1}(Y')}$ is controlled by the quantities named in the theorem.
\end{proof}

\section{Regularity of optimal maps and potentials}
\label{S:regularity}

Having found conditions guaranteeing smoothness of the husband's payoff $v(y)$ 
in the preceding section,  we now turn to the wife's payoff $u(x)$ and the optimal correspondence
$F:X\longrightarrow Y$ between wives and husbands.  

Our main conclusions are as follows: Lipschitz continuity 
of $F$ is equivalent to a strong form of
nestedness,  which requires a lower bound on the local speed of the motion of $F^{-1}(y)$ with respect
to changes in $y \in Y$.
On regions $F^{-1}([a,b])$ where this 
holds, higher regularity of $\f$ and $u$ (up to the boundary) are inherited from interior regularity
of $v$ via the first order conditions
$$
(Du(x),Dv(F(x))) = (D_x s(x,F(x)), D_ys(x,F(x))
$$
from \eqref{foc}.

We begin with a logically independent proposition, which will then be combined with results of the preceding section
to harvest the desired results.

\begin{proposition}
[Optimal maps have locally bounded variation]
\label{P:BV}
The hypotheses of Theorem \ref{T:nested}(c) and \ref{T:non-nested}(a)
imply $u \in C^1(X)$, $\f \in (BV_{loc}\cap C)(X)$,
$D_x s_y (\cdot,\f(\cdot)) \in (BV_{loc}\cap C) (X)$, $\f \in \Dom Dk$ on a set of $|D\f|$ full measure,
and 
\begin{equation}\label{distributional map gradient} 
(k'(\f(x))-s_{yy}(x,\f(x)))D\f(x) = D_x s_y(x,\f(x)).
\end{equation}
\end{proposition}

\begin{proof}
Since $k^+=k^-$ and $\log f$ and $\log g$ are bounded on compact subsets of $X$ and $Y$, 
we conclude $F$ is defined and coincides with
the continuous function $\bar F$ throughout $X$ from Theorem \ref{T:nested}(c).

It is well-known that the dual problem \eqref{dual} admits minimizers $(u,v)$ which inherit
Lipschitz and semiconvexity bounds from $s \in C^{2}(X\times Y)$ \cite{V}.  Then $u\in C^{0,1}(X)$ is differentiable Lesbesgue a.e.\ on $X$,   and 
\begin{equation}\label{ufoc}
Du(x) = D_x s(x,\f(x))
\end{equation}
for each $x\in \Dom Du$ according to \eqref{foc}.  The right hand side is continuous and bounded,
whence $u \in C^1(X)$.  

Since $u$ is also semiconvex,  its directional derivatives lie in $BV(X)$.
We next use \eqref{ufoc} to deduce $F \in BV_{loc}(X)$,  which means its directional
weak derivatives are signed Radon measures on $X$.  Fix $x' \in X$
and set $y'=F(x')\in Y$.  Since $D_x s_y(x',y')\ne 0$,  at least one of its components ---
say $\frac{\p^2 s}{\p x_1 \p y}(x,y)$ --- is non-vanishing in a neighbourhood of $(x',y')$.  The inverse function
theorem guarantees the map $(x,y) \mapsto (x,\frac{\p s}{\p x_1})$ has a continuously differentiable
inverse defined on a neighbourhood of $(x',\frac{\p s}{\p x_1}(x',y'))$.  From \eqref{ufoc} and the 
continuity of $\f$ and $Du$ we therefore deduce
$$
\f(x) = \left[ 
\frac{\p s}{\p x_1}(x,\cdot)\right]^{-1} \left(\frac{\p u}{\p x_1}(x)\right)
$$
expresses $\f$ as the composition of a $C^1$ map and a $BV$ map near $x'$.
This shows $\f \in BV_{loc}(X)$
\cite{AmbrosioDalMaso90}.

On the other hand, $k= \frac{dv}{dy} \in C^{0,1}_{loc}(Y)$ is locally Lipschitz 
according to Theorem \ref{T:non-nested}(a). 
According to \cite{AmbrosioDalMaso90}, $\f \in \Dom Dk$ on a set of $|D\f|$ full measure,
and differentiating 
$
k(\f(x)) = s_y(x,\f(x))
$
yields \eqref{distributional map gradient}
in the sense of measures; 
$D\f$ has no jump part since $\f$ is continuous.
\end{proof}

\begin{corollary}[Maps are $C^1$ where level set speed is non-zero]
\label{C:non-zero speed}
The hypotheses of Proposition \ref{P:BV} imply
$\| \f \|_{C^{0,1}(X')} \le \ell^{-1}\| D_x s_y \|_{C(X' \times F(X'))}<+\infty$
for any open set $X' \subset X$ having a speed limit
\begin{equation}\label{Clarke subdifferential bound}
\ell := \inf_{x \in X'} k'(\f(x)) -s_{yy}(x,\f(x)) >0.
\end{equation}
They also imply $\f$ is continuously differentiable on $X \setminus \f^{-1}(Z)$
at precisely those points  $x$ 
where $k'(\f(x)) > s_{yy}(x,\f(x))$.
\end{corollary}

\begin{proof}
The right hand side of \eqref{distributional map gradient} belongs to $C(X)$,
so the left hand side is also continuous and bounded.  
Thus we deduce 
$\ell \| Df \|_{L^\infty(X')} \le \| D_x s_y \|_{L^\infty(X' \times F(X'))} \le \|s\|_{C^{2}(X \times Y)} <+\infty$
as desired.

Since $k'(\f(x))-s_{yy}(x,\f(x))$ depends continuously on $x \in X \setminus \f^{-1}(Z)$,
\eqref{distributional map gradient} implies 
the same is true of $D\f(x)$ provided $k'(\f(x))>s_{yy}(x,\f(x))$.
Conversely, non-degeneracy of $D_x s_y$ prevents 
$D\f$ from being locally bounded 
where $k'(\f(x)) -s_{yy}(x,\f(x))$ vanishes.
\end{proof}

\medskip

We now proceeed to the main result of this section,
which uses condition \eqref{Clarke subdifferential bound}
to bootstrap higher regularity for $\f$ and $u$ 
from that already established for $v$ 

\begin{corollary}[Regularity of optimal maps and potentials]
\label{C:regularity}
Suppose $X' \subset X \subset \R^m$, $Y' \subset Y\subset \R$, $(s,f,g)$ and $k$ 
satisfy the hypotheses of Theorem \ref{T:h regular} for some $r \ge 1$.
If the speed limit condition $\ell>0$ from \eqref{Clarke subdifferential bound} holds, then
the restriction of the optimal map
$F:X \to Y$ to $X'$ lies in $C^{r,1}(X')$ 
and  $\|F\|_{C^{r,1}(X')}$ is controlled by $\ell$,
$\|s_y\|_{C^{r,1}(X' \times Y')}$ and $\|k\|_{C^{r,1}(Y')}$. 
Furthermore, if $s \in C^{r+1,1}(X' \times Y')$ then there exist minimizers $(u,v)$ 
of \eqref{dual} whose restrictions lie in $C^{r+1,1}(X') \times C^{r+1,1}(Y')$
with norms controlled by the same constants plus $\|s\|_{C^{r+1,1}(X' \times Y')}$.
\end{corollary}

\begin{proof}
Since $k=dv/dy$,  the regularity asserted for $v$ follows directly from Theorem \ref{T:h regular}.
That asserted for $\f$ then follows from the implcit function theorem after recalling that $k(F(x)) = s_y(x,F(x))$; this in turn implies that asserted 
for $u$ through \eqref{ufoc}.
\end{proof}
\begin{appendices}
\section{Self contained proof that nesting for all marginals reduces dimension}
We now offer a self contained proof of Proposition \ref{P:pseudo-index}.  Our proof requires an additional definition and preliminary lemma.
We say that \textit{the level sets of $x \mapsto \frac{\partial s}{\partial y}(x,y)$ are independent of $y$} if  for any $y_0,y_1 \in Y$, $S \subseteq X$ is a level set of $x \mapsto \frac{\partial s}{\partial y}(x,y_0)$ if and only if it is a level set of $x \mapsto \frac{\partial s}{\partial y}(x,y_1)$.  In other words, for each $k_0 \in \R$
there exists $k_1 \in \R$ such that $X(y_0,k_0)=X(y_1,k_1)$,  and conversely for each $k_1$ there exists $k_0$ 
satisfying the same conclusion.

\begin{lemma}\label{L:pseudo-index preparation}
Under the assumptions of Proposition \ref{P:pseudo-index}, the surplus $s$ takes pseudo-index form 
if and only if the level sets of $x \mapsto \frac{\partial s}{\partial y}(x,y)$ are independent of $y$.
 If $s$ has pseudo-index form then the mixed partials
$\frac{\p^2 \sigma}{\p I \p y} = \frac{\p^2 \sigma}{\p y \p I}$ exist and are continuous throughout $I(X) \times Y$.
\end{lemma}

\begin{proof}
If $s(x,y) = \sigma(I(x),y) +\alpha(x)$ is of  pseudo-index form, we have
\begin{equation}\label{c1}
\frac{\partial s}{\partial y}(x,y)=\frac{\partial \sigma}{\partial y}(I(x),y).
\end{equation}
Therefore, for every $y$, each level set of $I(x)$ is \emph{contained} in a level set of  $x \mapsto \frac{\partial s}{\partial y}(x,y)$.  If, in addition, we show that $I \mapsto \frac{\partial \sigma}{\partial y}(I,y)$ is injective, then the opposite inclusion will hold, and so the level sets of  $x \mapsto \frac{\partial s}{\partial y}(x,y)$ will be exactly the level sets of $I(x)$, which are clearly independent of $y$.

If $\sigma \in C^2$, then
\begin{equation}\label{c1c2}
D_x\frac{\partial s}{\partial y}(x,y) = \frac{\partial^2 \sigma}{\partial I \partial y}(I(x),y)D I(x),
\end{equation}
the non-degeneracy condition implies that $\frac{\partial^2 \sigma}{\partial I \partial y}$ is nowhere vanishing on $I(X) \times (a,b)$.  As the continuous image of a connected set, $I(X)$ is connected, and hence an interval;  we must therefore have either $\frac{\partial^2 \sigma}{\partial I \partial y}<0$ everywhere or $\frac{\partial^2 \sigma}{\partial I \partial y}>0$ everywhere (i.e., the Spence-Mirrlees sub- or super-modularity condition holds).  Therefore,
$$
I \mapsto \frac{\partial \sigma}{\partial y}(I,y)
$$
is injective, and so we conclude that the level sets of $x \mapsto \frac{\partial s}{\partial y}(x,y)$ are independent of $y$ .

If $\sigma \in C^1\setminus C^2$,  we shall use the non-degeneracy of $s\in C^2$ and
identity \eqref{pseudo-index} to argue that the mixed partials exist and are continuous,
in which case the analysis of the preceding paragraph applies. 
For fixed $x$, non-degeneracy implies
\begin{equation}\label{c1c2c3}
D_x s (x,y) = \frac{\p \sigma}{\p I} (I(x),y) DI(x)
\end{equation}
cannot vanish in any subinterval of $(a,b)$,  whence $DI(x) \ne 0$.  Using the inverse function theorem (e.g.
along an integral curve of the vector field $DI$),   we deduce differentiability of $\sigma_y(I,y)$ with respect to $I$, and continuous dependence of the resultant derivative on both $(I,y)$ from \eqref{c1} noting $s\in C^2$, thereby justifying \eqref{c1c2} and the resulting conclusions.  On the other hand,  we can also 
differentiate \eqref{c1c2c3} with respect to $y$ to obtain the other mixed partial, and
compare the result to \eqref{c1c2} to deduce the equality of mixed partials.

Conversely, suppose that the level sets of $x \mapsto \frac{\partial s}{\partial y}(x,y)$ are independent of $y$ and set $I(x) =\frac{\partial s}{\partial y}(x,\frac{a+b}2)$.   We claim that, for all $y$,

$$
x \mapsto s(x,y) -s(x,\frac{a+b}2)
$$
is constant along the level sets of $I(x)$.  This will complete the proof, as then we can define unambiguously the function $\sigma(z,y) :=s(x,y) -s(x,\frac{a+b}2)$, where $x \in I^{-1}(z)$, which implies
$$
s(x,y) =\sigma(I(x), y)+s(x,\frac{a+b}2)
$$
and hence $\sigma \in C^1$, noting continuity of $DI(x) = D_x s_y(x,\frac{a+b}2) \ne 0$.

To see the claim, let $x,\bar x$ lie in a level set of $I$, ie, $I(x) =I(\bar x)$; we need to show $h(y):= s(x,y) -s(x,\frac{a+b}2)-s(\bar x,y) +s(\bar x,\frac{a+b}{2})=0$ for all $y \in (a,b)$.    We clearly have $h(\frac{a+b}2) =0$.  On the other hand, differentiating yields
$$
h'(y) = \frac{\partial s}{\partial y}(x,y)- \frac{\partial s}{\partial y}(\bar x,y).
$$
Now $I(x) =I(\bar x)$ implies $h'(\frac{a+b}2)=0$. The assumed independence of the level sets therefore 
asserts $h'(y) =0$ for all $y \in (a,b)$,  yielding the desired result.
\end{proof}

We are now ready to prove the proposition.

\begin{proof}[Proof of Proposition \ref{P:pseudo-index}]
First assume $s$ has a  pseudo-index structure; by nondegeneracy, we can assume, without loss of generality, that $\frac{\partial^2 \sigma}{\partial I \partial y}>0$;  the derivatives in question exist by Lemma \ref{L:pseudo-index preparation}.
 In this case, for any probability measures $\mu$ and $\nu$, and $y_1>y_0 \in Y$, the same 
lemma implies that sets  $X_<(y_0,k^{\pm}(y_0)) $ and $X_<(y_1,k^{\pm}(y_1))$ correspond 
exactly to sublevel sets of $I(x)$:

$$
X_<(y_i,k^{\pm}(y_i)) =\{x: I(x) <c^{\pm}(y_i)\}
$$
where 
$[c^-(y),c^+(y)]$ is the maximal interval such that $I_{\#}\mu[(-\infty,z)] =\nu[(-\infty, y)]$ for all $z \in[c^-(y),c^+(y)]$.
  As we have 
\begin{equation}\label{eqn: nesting verification}
\mu \big [X_<(y_0,k^+(y_0)) \big ]=\nu[(0,y_0)] \leq \nu[(0,y_1)] =\mu \big[X_< (y_1,k^-(y_1)) \big ]
\end{equation}
this implies that $c^+(y_0) \leq c^-(y_1)$ and so
\begin{equation}\label{eqn: nesting set containment}
X_<(y_0,k^+(y_0)) \subset X_<(y_1,k^-(y_1)).
\end{equation}
If, in addition, $\nu[(y_0,y_1)] >0$, the inequality in \eqref{eqn: nesting verification} is strict, and so we must have  strict containment in \eqref{eqn: nesting set containment}.  Therefore, the model is nested.



On the other hand, assume that $s$ does not have a pseudo-index form; we will construct measures $\mu$ and $\nu$ for which nestedness fails.

We claim that there exist $y_0,y_1$, $k_0,k_1$ such that the pairwise disjoint, open sets 

\begin{eqnarray*}
A_{<<}& =&\{x \mid \frac{\partial s(x,y_0)}{\partial y} <k_0,  \frac{\partial s(x,y_1)}{\partial y} <k_1 \}\\
A_{<>}& =&\{x \mid \frac{\partial s(x,y_0)}{\partial y} <k_0,  \frac{\partial s(x,y_1)}{\partial y} >k_1 \}\\
A_{><}& =&\{x \mid \frac{\partial s(x,y_0)}{\partial y} >k_0,  \frac{\partial s(x,y_1)}{\partial y} <k_1 \}\\
A_{>>}& =&\{x \mid \frac{\partial s(x,y_0)}{\partial y} >k_0,  \frac{\partial s(x,y_1)}{\partial y} >k_1 \}\\
\end{eqnarray*}
are all  nonempty.  Their union is $X$, minus a set of Hausdorff dimension $n-1$, thanks to non-degeneracy and Proposition \ref{P:indifference structure}.

We first show that the claim implies the desired result, and then prove the claim.   We can assume, after rescaling, that $Y=(0,1)$.  Assuming without loss of generality that $y_1 >y_0$, for some small $\epsilon >0$, we can take $\mu$ to be a probability measure assigning the following values: 

\begin{eqnarray*}
\mu(A_{<<})&=&y_0-\epsilon,\\
 \mu(A_{<>})&=&\epsilon,\\ 
\mu(A_{><})&=&y_1-y_0+\epsilon,\\ 
\mu(A_{>>})&=&1-y_1-\epsilon. 
 \end{eqnarray*}

Taking $\nu$ to be uniform measure, the above choice of $\mu$ leads to $k(y_0) = k_0$ and $k(y_1)=k_1$.  In particular, this implies that the (nonempty) set  $A_{><} \subseteq X_{\leq}(y_1,k(y_1))$ but that $A_{><}$ is \emph{disjoint} from  $X_{\leq}(y_0,k(y_0))$, violating nestedness.

  To see the claim, we invoke the lemma to conclude that there exists $y_0, y_1$ and $k_0$, such that the  level set 
$$
X(y_0,k_0)= \{x \mid \frac{\partial s(x,y_0)}{\partial y}=k_0\}
$$
is not a level set of $x \mapsto \frac{\partial s(x,y_1)}{\partial y}$.  Choose any $\bar x \in X(y_0,k_0)$ and set $k_1 =  \frac{\partial s(\bar x,y_1)}{\partial y}$ and $X(y_1,k_1) :=\{x \mid \frac{\partial s( x,y_1)}{\partial y}=k_1\}$.  As  we cannot have

$$
X(y_0,k_0) =X(y_1,k_1),
$$
we assume, without loss of generality, that there exists $x \in X(y_1,k_1)$ such that $\frac{\partial s( x,y_0)}{\partial y} <k_0$.

Now, if the $C^1$ hypersurfaces $X(y_1,k_1)$  and $X(y_0,k_0)$ interesect tranversally at $\bar x$, the result follows easily: in this case   $ X(y_1,k_1)$ must intersect both $X_<(y_0,k_0)$ and $X_>(y_0,k_0)$ and at the intersection points, one can move slightly off the set $ X(y_1,k_1)$, into either of $ X_<(y_1,k_1)$ or $ X_>(y_1,k_1)$, and remain in  $X_<(y_0,k_0)$ or $X_>(y_0,k_0)$.

If not, the intersection in tangential at $\bar x$.  The nondegeneracy condition ensures that we may choose $k_1'$ close to $k_1$ such that  the set $X(y_1,k_1') $ intersects $X_{>}(y_0,k_0)$ near $\bar x$ (simply by moving a small distance from $\bar x$ in the direction $D_x\frac{\partial s}{\partial y}(\bar x, y_0)$ to obtain a new point $\bar x'$ and setting $k_1' = \frac{\partial s}{\partial y}(\bar x', y_1)$).  This implies that $A_{><}$ and $A_{>>}$ are both nonempty, after replacing $k_1$ with $k_1'$.  Nondegeneracy and the implicit function theorem  also implies that $X(y_1,k_1')$ intersects  $X_{<}(y_0,k_0)$ near $x$, which implies that $A_{<<}$ and $A_{<>}$ are both nonempty, completing the proof of the claim.
\end{proof}

\end{appendices}

\def\cprime{$'$}


\begin{thebibliography}{10}

\bibitem{akm}
Najma Ahmad, Hwa~Kil Kim, and Robert~J McCann.
\newblock Optimal transportation, topology and uniqueness.
\newblock {\em Bull. Math. Sci.}, 1:13--32, 2011.

\bibitem{AmbrosioDalMaso90}
Luigi Ambrosio and Gianni Dal~Maso.
\newblock A general chain rule for distributional derivatives.
\newblock {\em Proc. Amer. Math. Soc.}, 108(3):691--702, 1990.

\bibitem{Becker73}
Gary~S Becker.
\newblock {A theory of marriage. Part I}.
\newblock {\em J. Political Econom.}, 81:813--846, 1973.

\bibitem{bren}
Yann Brenier.
\newblock Decomposition polaire et rearrangement monotone des champs de
  vecteurs.
\newblock {\em C.R. Acad. Sci. Pair. Ser. I Math.}, 305:805--808, 1987.

\bibitem{Caffarelli92}
Luis~A Caffarelli.
\newblock The regularity of mappings with a convex potential.
\newblock {\em J. Amer. Math. Soc.}, 5:99--104, 1992.

\bibitem{Caffarelli96b}
Luis~A Caffarelli.
\newblock Boundary regularity of maps with convex potentials --- {II}.
\newblock {\em Ann. of Math. (2)}, 144:453--496, 1996.

\bibitem{cmn}
Pierre-Andre Chiappori, Robert~J McCann, and Lars Nesheim.
\newblock Hedonic price equilibria, stable matching and optimal transport;
  equivalence, topology and uniqueness.
\newblock {\em Econom. Theory.}, 42(2):317--354, 2010.

\bibitem{ChiapporiMcCannPass15p}
Pierre-Andre Chiappori, Robert~J McCann, and Brendan Pass.
\newblock Multidimensional matching.
\newblock {\em Preprint}.

\bibitem{Clarke83}
Frank~H Clarke.
\newblock {\em Optimization and nonsmooth analysis}.
\newblock Wiley, New York, 1983.

\bibitem{CuestaMatran89}
Juan~Antonio Cuesta-Albertos and Carlos Matr\'an.
\newblock Notes on the {W}asserstein metric in {H}ilbert spaces.
\newblock {\em Ann. Probab.}, 17:1264--1276, 1989.

\bibitem{Cullen06}
Michael~J Cullen.
\newblock {\em A Mathematical Theory of Large Scale Atmosphere/Ocean Flows}.
\newblock Imperial College Press, London, 2006.

\bibitem{CullenPurser84}
Michael~JP Cullen and R~James Purser.
\newblock An extended {Lagrangian} model of semi-geostrophic frontogenesis.
\newblock {\em J. Atmos. Sci.}, 41:1477--1497, 1984.

\bibitem{DePhilippisFigalli14}
Guido De~Philippis and Alessio Figalli.
\newblock The {M}onge-{A}mp\`ere equation and its link to optimal
  transportation.
\newblock {\em Bull. Amer. Math. Soc. (N.S.)}, 51(4):527--580, 2014.

\bibitem{EvansGariepy92}
Lawrence~C Evans and Ronald~F Gariepy.
\newblock {\em Measure Theory and Fine Properties of Functions}.
\newblock Stud. Adv. Math. CRC Press, Boca Raton, 1992.

\bibitem{Federer69}
Herbert Federer.
\newblock {\em Geometric Measure Theory}.
\newblock Springer-Verlag, New York, 1969.

\bibitem{FigalliKimMcCannARMA13}
Alessio Figalli, Young-Heon Kim, and Robert~J McCann.
\newblock {H\"older continuity and injectivity of optimal maps}.
\newblock {\em Arch. Rational Mech. Anal.}, 209:747--795, 2013.

\bibitem{FrischMatarreseMohayaeeSobolevskii02}
Uriel Frisch, Sabino Matarrese, Roya Mohayaee, and Andrei Sobolevskii.
\newblock A reconstruction of the initial conditions of the universe by optimal
  mass transportation.
\newblock {\em Nature}, 417:260--262, 2002.

\bibitem{Galichon16p}
Alfred Galichon.
\newblock {\em Optimal Transport Methods in Economics}.
\newblock Princeton University Press, to appear.

\bibitem{G}
Wilfrid Gangbo.
\newblock Habilitation thesis, Universite de Metz, available at
  http://people.math.gatech.edu/~gangbo/publications/habilitation.pdf, 1995.

\bibitem{GM}
Wilfrid Gangbo and Robert~J McCann.
\newblock The geometry of optimal transportation.
\newblock {\em Acta Math.}, 177:113--161, 1996.

\bibitem{GM2}
Wilfrid Gangbo and Robert~J McCann.
\newblock Shape recognition via {W}asserstein distance.
\newblock {\em Quart. Appl. Math.}, 58(4):705--737, 2000.

\bibitem{Gigli11}
Nicola Gigli.
\newblock {On the inverse implication of Brenier-McCann theorems and the
  structure of $(P_2(M),W_2)$}.
\newblock {\em Methods Appl. Anal.}, 18:127--158, 2011.

\bibitem{GuilleminPollack74}
Victor Guillemin and Alan Pollack.
\newblock {\em Differential Topology}.
\newblock Prentice-Hall, Englewood Cliffs, NJ, 1974.

\bibitem{HakerTannenbaum03}
Steven Haker and Allen Tannenbaum.
\newblock Optimal image interpolation and optical flow.
\newblock In {\em Multidisciplinary research in control ({S}anta {B}arbara,
  {CA}, 2002)}, volume 289 of {\em Lecture Notes in Control and Inform. Sci.},
  pages 133--143. Springer, Berlin, 2003.

\bibitem{HofmannMitreaTaylor10}
Steve Hofmann, Marius Mitrea, and Michael Taylor.
\newblock {Singular integrals and elliptic boundary problems on regular
  Semmes-Kenig-Toro domains}.
\newblock {\em Int. Math. Res. Not. IMRN}, 14:2567--2865, 2010.

\bibitem{Kant}
Leonid Kantorovich.
\newblock On the translocation of masses.
\newblock {\em C.R. (Doklady) Acad. Sci. URSS (N.S.)}, 37:199--201, 1942.

\bibitem{KitagawaWarren12}
Jun Kitagawa and Micah Warren.
\newblock Regularity for the optimal transportation problem with euclidean
  distance squared cost on the embedded sphere.
\newblock {\em SIAM J. Math. Anal.}, 44:2871--2887, 2012.

\bibitem{KnottSmith84}
Martin Knott and Cyril~S Smith.
\newblock On the optimal mapping of distributions.
\newblock {\em J. Optim. Theory Appl.}, 43:39--49, 1984.

\bibitem{Knuth97}
Donald~E Knuth.
\newblock {\em Stable marriage and its relation to other combinatorial
  problems}, volume~10 of {\em CRM Proceedings \& Lecture Notes}.
\newblock American Mathematical Society, Providence, RI, 1997.
\newblock An introduction to the mathematical analysis of algorithms,
  Translated from the French by Martin Goldstein and revised by the author.

\bibitem{lev}
Vladimir~L Levin.
\newblock {Abstract cyclical monotonicity and Monge solutions for the general
  Monge-Kantorovich problem}.
\newblock {\em Set-Valued Analysis}, 7(1):7--32, 1999.

\bibitem{LiuTrudingerWang10}
Jiakun Liu, Neil~S Trudinger, and Xu-Jia Wang.
\newblock Interior {$C^{2,\alpha}$} regularity for potential functions in
  optimal transportation.
\newblock {\em Comm. Partial Differential Equations}, 35:165--184, 2010.

\bibitem{loeper}
Gregoire Loeper.
\newblock On the regularity of solutions of optimal transportation problems.
\newblock {\em Acta Math.}, 202:241--283, 2009.

\bibitem{Lorentz53}
George~G Lorentz.
\newblock An inequality for rearrangements.
\newblock {\em Amer. Math. Monthly}, 60:176--179, 1953.

\bibitem{mtw}
Xi-Nan. Ma, Neil~S Trudinger, and Xu-Jia Wang.
\newblock Regularity of potential functions of the optimal transportation
  problem.
\newblock {\em Arch. Rational Mech. Anal.}, 177:151--183, 2005.

\bibitem{McCann95}
Robert~J McCann.
\newblock Existence and uniqueness of monotone measure-preserving maps.
\newblock {\em Duke Math. J.}, 80:309--323, 1995.

\bibitem{McCann97}
Robert~J McCann.
\newblock A convexity principle for interacting gases.
\newblock {\em Adv. Math.}, 128:153--179, 1997.

\bibitem{Mc14}
Robert~J McCann.
\newblock A glimpse into the differential topology and geometry of optimal
  transport.
\newblock {\em Discrete Contin. Dyn. Syst.}, 34:1605--1621, 2014.

\bibitem{McCannRifford15p}
Robert~J McCann and Ludovic Rifford.
\newblock {The intrinsic dynamics of optimal transport}.
\newblock To appear in {\em J. \'Ec. polytech. Math.} 2016. 

\bibitem{Mirrlees71}
James~A Mirrlees.
\newblock An exploration in the theory of optimum income taxation.
\newblock {\em Rev. Econom. Stud.}, 38:175--208, 1971.

\bibitem{Monge81}
Gaspard Monge.
\newblock M\'emoire sur la th\'eorie des d\'eblais et de remblais.
\newblock {\em Histoire de l'{A}cad\'emie Royale des Sciences de {P}aris, avec
  les {M}\'emoires de Math\'ematique et de Physique pour la m\^eme ann\'ee},
  pages 666--704, 1781.

\bibitem{P2}
Brendan Pass.
\newblock Regularity of optimal transportation between spaces with different
  dimensions.
\newblock {\em Math. Res. Lett.}, 19(2):291--307, 2012.

\bibitem{RachevRuschendorf98}
Svetlozar~T Rachev and L{\"u}dger R{\"u}schendorf.
\newblock {\em Mass Transportation Problems}.
\newblock Probab. Appl. Springer-Verlag, New York, 1998.

\bibitem{Rochet87}
Jean-Charles Rochet.
\newblock A necessary and sufficient condition for rationalizability in a
  quasi-linear context.
\newblock {\em J. Math. Econom.}, 16:191--200, 1987.

\bibitem{RochetChone98}
Jean-Charles Rochet and Philippe Chon\'e.
\newblock Ironing, sweeping and multidimensional screening.
\newblock {\em Econometrica}, 66:783--826, 1998.

\bibitem{RuschendorfRachev90}
L{\"u}dger R{\"u}schendorf and Svetlozar~T Rachev.
\newblock A characterization of random variables with minimum {$L^2$}-distance.
\newblock {\em J. Multivariate Anal.}, 32:48--54, 1990.

\bibitem{Santambrogio15p}
Filippo Santambrogio.
\newblock {\em Optimal Transport for Applied Mathematicians}.
\newblock {\em Preprint}, 2015.

\bibitem{SS}
Lloyd~S Shapley and Martin Shubik.
\newblock {The assignment game I: The core}.
\newblock {\em International Journal of Game Theory}, 1(1):111--130, 1971.

\bibitem{Spence73}
Michael Spence.
\newblock Job market signaling.
\newblock {\em Quarterly J. Econom.}, 87:355--374, 1973.

\bibitem{TrudingerWang09b}
Neil~S Trudinger and Xu-Jia Wang.
\newblock {On the second boundary value problem for Monge-Amp\`ere type
  equations and optimal transportation}.
\newblock {\em Ann. Sc. Norm. Super. Pisa Cl. Sci. (5)}, 8:1--32, 2009.

\bibitem{V}
C\'edric Villani.
\newblock {\em Topics in optimal transportation}, volume~58 of {\em Graduate
  Studies in Mathematics.}
\newblock American Mathematical Society, Providence, 2003.

\bibitem{V2}
C\'edric Villani.
\newblock {\em Optimal transport: old and new}, volume 338 of {\em Grundlehren
  der mathematischen Wissenschaften.}
\newblock Springer, New York, 2009.

\end{thebibliography}
\end{document}